\documentclass[a4paper, 11pt, twoside, notitlepage]{amsart}

\usepackage{amsmath,amscd}
\usepackage{amssymb}
\usepackage{amsthm}
\usepackage{comment}
\usepackage{graphicx, xcolor}

\usepackage{mathrsfs}
\usepackage[ocgcolorlinks, linkcolor=blue]{hyperref}
\usepackage[normalem]{ulem}

\newcommand\redsout{\bgroup\markoverwith{\textcolor{red}{\rule[0.5ex]{2pt}{0.4pt}}}\ULon}

\usepackage{bm}
\usepackage{bbm}
\usepackage{url}

\usepackage[most]{tcolorbox}

\usepackage[utf8]{inputenc}
\usepackage{mathtools,amssymb}
\usepackage{esint}
\usepackage{tikz}
\usepackage{dsfont}
\usepackage{relsize}
\usepackage{url}
\usepackage{xcolor}
\usepackage{graphicx}
\usepackage{mathrsfs}
\usepackage[shortlabels]{enumitem}
\usepackage{lineno}
\usepackage{amsmath}
\usepackage{enumitem}
\usepackage{amsthm} 
\usepackage{verbatim}
\usepackage{dsfont}
\numberwithin{equation}{section}

\allowdisplaybreaks

\mathtoolsset{showonlyrefs}

\graphicspath{{images/}}

\newtheorem{theorem}{Theorem}[section]
\newtheorem{lemma}[theorem]{Lemma}
\newtheorem{definition}[theorem]{Definition}

\newtheorem{proposition}[theorem]{Proposition}

\newtheorem{assumption}{Assumption}
\newtheorem{remark}[theorem]{Remark}

\usepackage{url}
\definecolor{mycolor}{rgb}{0.122, 0.435, 0.698}
\definecolor{aliceblue}{rgb}{0.94, 0.97, 1.0}

\title[Entanglement principle for the nonlocal parabolic operators]{Entanglement principle and fractional Calder\'on problem for nonlocal parabolic operators}

\author[R.-Y. Lai]{Ru-Yu Lai}
\address{School of Mathematics, University of Minnesota, Minneapolis, MN 55455, USA}
\curraddr{}
\email{rylai@umn.edu}

\author[Y.-H. Lin]{Yi-Hsuan Lin}
\address{Department of Applied Mathematics, National Yang Ming Chiao Tung University, Hsinchu, Taiwan \& Fakult\"at f\"ur Mathematik, University of Duisburg-Essen, Essen, Germany}
\curraddr{}
\email{yihsuanlin3@gmail.com}

\author[L. Yan]{Lili Yan}
\address{Department of Mathematics,  North Carolina State University,, Raleigh, NC 27695, USA}
\curraddr{}
\email{lyan6@ncsu.edu}


\newcommand{\C}{{\mathbb C}}
\newcommand{\R}{{\mathbb R}}
\newcommand{\Z}{{\mathbb Z}}
\newcommand{\N}{{\mathbb N}}

\newcommand{\eps}{\epsilon}

\newcommand {\p} {\partial}

\newcommand{\LC}{\left(}
\newcommand{\RC}{\right)}

\newcommand{\norm}[1]{\lVert #1 \rVert}
\newcommand{\abs}[1]{\left\lvert #1 \right\rvert}
\DeclareMathOperator{\supp}{supp} 
\DeclareMathOperator{\dist}{dist} 

\begin{document}

	\maketitle
	
	\begin{abstract}
		We examine inverse problems for the variable-coefficient nonlocal parabolic operator $(\partial_t - \Delta_g)^s$, where $0 < s < 1$. This article makes two primary contributions. First, we introduce a novel entanglement principle for these operators under suitable smoothness conditions. Second, we prove that lower-order perturbations can be uniquely determined from the associated Dirichlet-to-Neumann map using this principle. However, due to insufficient solution regularity, direct application of the entanglement principle to the inverse problem is not feasible. To address this, we derive a modified entanglement principle, enabling the effective resolution of related inverse problems.
	\end{abstract}
	
	\tableofcontents
	
	\section{Introduction}\label{sec: introduction}
	
	Inverse problems for space-fractional equations have garnered considerable attention in recent years, not only due to their distinctive mathematical features, but also because of their wide range of applications in physics, biology, finance, and related fields.  
	
	A pioneering breakthrough in this direction is the resolution of the Calderón problem for the fractional Schrödinger equation (see \cite{GSU20}), which concerns the recovery of an unknown bounded potential from exterior measurements. One of the key contributions in \cite{GSU20} is the establishment of the \emph{unique continuation property} (UCP) for the fractional Laplacian operator $(-\Delta)^s$, $0<s<1$, which states that  
	\begin{equation}
		u = (-\Delta)^s u = 0 \quad \text{in a nonempty open subset of } \R^n  \implies  u \equiv 0 \text{ in } \R^n.
	\end{equation}
	This fundamental property leads to the \emph{Runge approximation property}, which asserts that any $L^2$-function on a given open set can be approximated by solutions of the fractional Schr\"odinger equation $\big((-\Delta)^s + q\big)u = 0$. Either UCP or Runge approximation can then be used to show that a lower-order perturbation $q$ can be uniquely determined from exterior data, under suitable regularity conditions on $q$.  
	
	Following this seminal work \cite{GSU20}, a substantial body of research has emerged on inverse problems for various space-fractional models. For example, simultaneous recovery results for multiple parameters were obtained in \cite{CLL2017simultaneously, cekic2020calderon}, while the determination of bounded potentials for anisotropic nonlocal Schr\"odinger equations was investigated in \cite{GLX}. These problems remain open in the local case $s=1$ for dimensions $n \geq 3$, suggesting that nonlocality appears to provide significant advantages in addressing such inverse problems. For further developments on both linear and nonlinear nonlocal inverse problems in various settings, we refer the reader to the following articles \cite{harrach2017nonlocal-monotonicity, harrach2020monotonicity, LL2020inverse, GRSU20, CMRU20, RS20, ruland2018exponential, GRSU20, LLR2019calder, LZ2023unique, GU2021calder} and the references therein.

	In particular, owing to the close connection between nonlocal and local settings, interior coefficients can be recovered either via reductions based on the Caffarelli--Silvestre extension (see, e.g., \cite{CGRU2023reduction,ruland2023revisiting, LLU2023calder, LZ2024approximation}) or by employing heat semigroup methods on closed Riemannian manifolds (see, e.g., \cite{feizmohammadi2021fractional, Fei24_TAMS, FKU24,lin2024fractional}). Both approaches share a similar philosophy: they transfer certain nonlocal information to its local counterpart, or to the associated heat equation. In contrast, a local-to-nonlocal reduction was studied in \cite{LNZ24} for the classical Schr\"odinger equation in transversal anisotropic geometry.  
	
	More recently, the study of perturbation by nonlocal operators has gathered interest because of their intrinsic mathematical properties and potential applications. The corresponding unique continuation property, referred to as the \emph{entanglement principle}, was established in \cite{FKU24} for the fractional Laplace--Beltrami operator and in \cite{FL24} for the fractional Laplacian. Its remarkable capacity to disentangle contributions from each fractional power in a nonlocal operator is expected to inspire further developments in nonlocal inverse problems. For a comprehensive survey of this rapidly evolving field, we refer the reader to the recent monograph \cite{LL25_Integro}.

	\subsection{Mathematical formulations and main results}

	In this work, we focus on inverse problems associated with fractional nonlocal parabolic operators. Let's begin by defining the parabolic operator
	\begin{equation}\label{para op.}
		\mathcal{H}_g:= \p_t -\Delta_g,
	\end{equation} 
	where the Laplace--Beltrami operator $\Delta_g$ is defined as follows:
	$$
	\Delta_g={1\over \sqrt{|g|}} \sum^n_{j,k=1}{\p\over \p x_j} \LC\sqrt{|g|}g^{jk}{\p\over \p x_k}\RC.
	$$  
	Here the metric $g=\LC g_{jk}(x)\RC_{1\leq j,k\leq n}\in C^\infty (\R^n;\R^{n\times n})$ satisfies the ellipticity condition, i.e., there exists a constant $\lambda \in (0,1)$ such that
	\begin{equation}\label{ellipticity}
		\lambda |\xi|^2 \leq \sum_{j,k=1}^n g_{jk}(x)\xi_j \xi_k \leq \lambda^{-1}|\xi|^2,
	\end{equation}
	for any $x\in \R^n$ and for any $\xi=(\xi_1,\ldots, \xi_n) \in \R^n$. Also, $|g|$ stands for the absolute value of the determinant of $g$, and $g^{jk}$ are the components of the inverse of $g=\LC g_{jk}(x)\RC_{1\leq j,k\leq n}$.

	The fractional powers of the parabolic operator $\mathcal{H}_g$ is defined by 
	$$
	\mathcal{H}_g^su : = (\p_t -\Delta_g)^s u,\quad s>0,
	$$
	of a function $u=u(x,t):\R^{n+1}\rightarrow \R$, $n\in \N$. These space-time nonlocal operators are found in various applications, including continuous-time random walks and mathematical biology. In contrast with the operators like $\p_tu+ (-\Delta_g)^s$, the space and time variables in $\mathcal{H}^s_g$ are coupled together with order $s$ in time $t$ and order $2s$ in space $x$. 
	
	To define such operators, when $g=\mathrm{I}_n$  
	(the $n\times n $ identity matrix), we simply write $\mathcal{H}:= \p_t -\Delta$ as a heat operator, where $\Delta$ denotes the classical Laplace operator. 
	In this case, since the coefficients of $\mathcal{H}$ are constants, the operator $\mathcal{H}^s$ can be defined via the Fourier transform as follows:
	\begin{align}
		\label{eq:op}
		\widehat{\LC \mathcal{H}^su\RC}(\xi,\rho)=(i\rho+|\xi|^2)^s\widehat{u}(\xi,\rho), \quad \text{ for }(\xi,\rho)\in \R^n\times\R,
	\end{align}
	for $u\in \mathcal{S}(\R^{n+1})$, which is the Schwartz space of smooth, rapidly decreasing functions in $\R^n\times\R$. Here $\widehat{f}$ denotes the Fourier transform of $f$ in both space and time variables $(x,t)$.
	For the latter purpose, we also use $\mathcal{F}_x(f)$ and $\mathcal{F}_t(f)$ to denote the Fourier transform of $f$ with respect to $x$ and $t$, respectively. However, using the method of the Fourier transform to define the operator $\mathcal{H}^s_g$ with variable coefficients is not directly applicable. Fortunately, another more flexible approach  through the parabolic language of semigroups is available to handle $\mathcal{H}^s_g$, see for instance,   
	\cite{ST17_para} and also the discussion in Section~\ref{sec: prel}.
	In particular, it can also provide the explicit formula for $\mathcal{H}^s_g$ at the space-time point $(x,t)$.

	\subsubsection{Previous work on nonlocal parabolic inverse problems}
	Let us revisit 
	inverse problems related to nonlocal parabolic operators.
	We consider the Calder\'{o}n-type problem for the equation $\LC \mathcal{H}_g^s +V \RC  u  =0 $ with a \emph{time-dependent} potential $V=V(x,t)$. The goal is to recover $V$ from exterior measurements $\Lambda_V$ defined below. 
	
	Let $u_f$ be the solution to the problem
	\begin{equation}\label{equ: nonlocal para 1}
		\begin{cases}
			\LC \mathcal{H}_g^s +V \RC  u  =0 &\text{ in }\Omega_T, \\
			u=f &\text{ in }(\Omega_e)_T, \\
			u=0 &\text{ in }\R^n \times \{t\leq -T\},
		\end{cases}
	\end{equation}
	where we use the standing notations
	\begin{equation}
		\Omega_e : = \R^n \setminus \overline{\Omega}, 
	\end{equation}
	and 
	\begin{equation}
		A_T := A\times (-T,T),
	\end{equation}
	for any subset $A\subset \R^n$ and a fixed real number $T>0$ throughout this work. Notice that our initial condition in \eqref{equ: nonlocal para 1} is needed due to the natural nonlocality of the operator $\mathcal{H}^s_g$.
	
	To study inverse problems for \eqref{equ: nonlocal para 1}, we require an additional eigenvalue condition: Suppose that $\{0\}$ is not a Dirichlet eigenvalue of \eqref{equ: nonlocal para 1} in the sense that 
	\begin{equation}\label{eigenvalue condition}
		\begin{cases}
			\text{If $u\in \mathbf{H}^s(\R^{n+1})$ solves } \begin{cases}
				\big(  \mathcal{H}^s_g +V \big) u =0 &\text{ in }\Omega_T, \\
				u=0& \text{ in }(\Omega_e)_T,\\
				u=0 &\text{ in }\R^n \times \{t\leq -T\},
			\end{cases}\\
			\text{then }u\equiv 0 \text{ in }\R^n_T.
		\end{cases}
	\end{equation}
	Here the space $\mathbf{H}^s(\R^{n+1})$ is defined in Section~\ref{sec: prel}. It is well-known that for all bounded potentials $V\geq 0$, \eqref{eigenvalue condition} is satisfied automatically.
	Since the condition \eqref{eigenvalue condition} ensures unique solvability of the forward problem for \eqref{equ: nonlocal para 1} (see Section~\ref{sec: prel}), we can formally define the \emph{Dirichlet-to-Neumann} (DN) map $\Lambda_V$ of \eqref{equ: nonlocal para 1} by
	\begin{equation}
		\Lambda_V : \widetilde{\mathbf{H}}^s((\Omega_e)_T) \to \mathbf{H}^{-s}((\Omega_e)_T), \quad f\mapsto \left. \mathcal{H}^s_g u_f \right|_{(\Omega_e)_T}.
	\end{equation}
	By using the DN map $\Lambda_V$, it has been shown in \cite{LLR2019calder, BS2024calderon} that the time-dependent potential $V$ can be recovered uniquely.
	
	We would also like to point out that in the works \cite{LLU2022para, LLU2023calder}, the authors show the nonlocal-to-local reduction with respect to both the heat semigroup and the Caffarelli-Silvestre extension approaches. More precisely, they demonstrated that the nonlocal DN map for nonlocal parabolic equations can determine their local DN map for local parabolic equations.

	\subsubsection{Main results}
	
	In this article, we extend our study to the models with the fractional poly-parabolic operators $\sum_{k=1}^N b_k \mathcal{H}^{\alpha_k}_g$ for constants $b_k$. 
	As mentioned above, several previous works have been devoted to studying inverse problems for the operator $\mathcal{H}^s_g$ with local perturbation. However, when nonlocal operators perturb it, the complicated nonlocal interactions contributed from different terms make the problem challenging to solve. To decouple their entangled effect, we introduce a novel approach in Theorem~\ref{thm: ent} and use it to study the related inverse problems in Theorem~\ref{thm: IP_ent}.

	Now we state the setting of the problem to be investigated and the key approach. Given an integer $N\geq 2$, let $\mathcal{O}\subset \R^n$ be a nonempty open set. Suppose that $u_k\in C^\infty(\R^{n+1})$, for $k=1,\ldots, N$. Inspired by the papers \cite{FKU24, FL24} which address the entanglement issue for fractional elliptic operators, we are interested in the following question: 
	\begin{enumerate}[\textbf{(IP-1)}]
		\item  \label{question} \textit{Given $N\geq 2$.
			Suppose that $\{ u_k \}_{k=1}^N \subset C^\infty(\R^{n+1})$ satisfies 
			\begin{equation}
				\left. u_1 \right|_{\mathcal{O}_T}  = \ldots =  \left. u_N \right|_{\mathcal{O}_T} = \bigg(\sum_{k=1}^N b_k \mathcal{H}^{\alpha_k}_g u_k \bigg) \bigg|_{\mathcal{O}_T}=0, 
			\end{equation}
			where $\{ b_k\} \subset \C\setminus \{0\}$, and $\{\alpha_k\}\subset (0,\infty)\setminus \N$ are given real numbers.
			Does there hold $u_k\equiv 0 $ in $\R^{n}_T$ for all $k=1,\ldots, N$?}
	\end{enumerate} 
	The notation $\mathbb{N}$ denotes the set of all positive integers, and $\mathbb{Z}$ denotes the set of all integers. Note that when $N=1$, for a single operator $\mathcal{H}^s_g$, this property is referred to as the \emph{unique continuation property} (UCP) and has been studied in \cite{LLR2019calder, BS2024calderon}.

	\begin{remark}
		Restrictions on the exponents $\{\alpha_k\}_{k=1}^N$ are necessary, since for general choices of $\{\alpha_k\}_{k=1}^N$, Question \ref{question} does not always hold. To illustrate this, consider the case $N=2$. Let $\alpha \in (0,\infty)\setminus \N$ and set $\alpha_1 = \alpha + m$, $\alpha_2 = \alpha$ for some $m \in \N$. Given a nonempty open subset $\mathcal{O}\subset \R^n$, let $u_1 \in C^\infty(\R^n_T)$ be a nontrivial function such that $u_1=0$ in $\mathcal{O}_T$.  
		Since $\mathcal{H}_g^m$ is a local operator for $m \in \N$, by defining $u_2 := -\mathcal{H}_g^m u_1$, it follows that $u_2=0$ in $\mathcal{O}_T$ and, moreover, the equation  
		\begin{equation}
			\mathcal{H}_g^{\alpha_1}u_1 + \mathcal{H}_g^{\alpha_2}u_2 = 0 \quad \text{in }\mathcal{O}_T.
		\end{equation}
		However, $u_1$ and $u_2$ are not trivial functions.
	\end{remark}
	
	The above counterexample shows that no such principle exists for local operators, and it leads to the following optimal condition for our entanglement principle:
	
	\begin{assumption}\label{exponent condition} 
		We assume $\{\alpha_k\}_{k=1}^N \subset (0,\infty)\setminus \N$ with $\alpha_1<\alpha_2<\ldots<\alpha_N$, and that they satisfy 
		\begin{equation}\label{exponent condition for ent prin}
			\alpha_k-\alpha_j \notin \Z \quad \text{for all } j\neq k,
		\end{equation}
		which is required to ensure a positive answer to Question \ref{question}. 
		This nonresonance condition guarantees that the fractional powers remain genuinely distinct and cannot be reduced to local operators by integer shifts.
	\end{assumption}
	
	Our first main result, which decouples the entangled effects, is stated as follows.

	\begin{theorem}[Entanglement principle]\label{thm: ent}
		Let $\mathcal{O}\subset \R^n$ be a nonempty open set for $n\geq 2$. Let $N\in\mathbb{N}$, $T>0$, and $\{\alpha_k\}_{k=1}^N \subset (0,\infty)\setminus \N $ satisfy \textbf{Assumption~\ref{exponent condition}}. Suppose $g \in C^\infty(\R^n;\R^{n\times n})$ satisfy \eqref{ellipticity}. Assume that $\{u_k\}_{k=1}^{N}\subset C^\infty((-\infty,T);\mathcal{S}(\R^n))$ 
		satisfy the following estimates: given any $\beta=(\beta_0,\beta_1,\ldots,\beta_n) \in \LC \N \cup \{0\} \RC^{n+1}$, there exist positive constants $C_\beta$ and $\delta$ such that   
		\begin{equation}\label{solution exponential decay time}
			\begin{split}
				\big|D^\beta_{x,t} u_k(x,t)\big|\leq 
				C_\beta|\varphi_\beta(x)|e^{\delta t},\quad \quad |\beta|\ge 0
				\quad &\text{for }  (x,t)\in \R^n\times \{t\leq -T\},
			\end{split}
		\end{equation}
		for $k=1,\ldots, N$, where $D_{x,t}^\beta = \frac{\p^{|\beta|}}{\p_t^{\beta_0}\p x_1^{\beta_1}\ldots\, \p x_n^{\beta_n}}$ and $\varphi_\beta \in \mathcal{S}(\R^n)$. 
		If
		\begin{equation}\label{condition_entanglement_u}
			\begin{split}
				\left. u_1\right|_{\mathcal O_T}=\ldots=u_N|_{\mathcal O_T}=0 \quad \text{and} \quad  \bigg(  \sum_{k=1}^N  b_k\mathcal{H}_g^{\alpha_k}u_k\bigg) \bigg|_{\mathcal O_T}=0 
			\end{split}
		\end{equation}
		hold for some $\{b_k\}_{k=1}^N\subset \C\setminus \{0\}$, then $u_k\equiv 0$ in $\R^n_T$ for all $k=1,\ldots,N$.
	\end{theorem}
	Theorem \ref{thm: ent} extends the entanglement principle established in \cite{FKU24, FL24} for elliptic operators to the parabolic setting. Specifically, \cite{FKU24} proved the validity of the entanglement principle for the Laplace--Beltrami operator in the compact case, while \cite{FL24} addressed the non-compact case for the classical Laplace operator.
	
	We also would like to emphasize that the decay condition \eqref{solution exponential decay time} will not impose any additional assumption in the study of inverse problem in Theorem~\ref{thm: IP_ent}. Moreover, when $N=1$, \eqref{solution exponential decay time} can be removed in Theorem~\ref{thm: ent}, see Remark~\ref{remark:entangle section 3} for more detailed discussions.
	
	\begin{remark}
		In our case, owing to the extra time variable in the fractional parabolic equation, the exponents $\{\alpha_k\}_{k=1}^N$ only need to satisfy \textbf{Assumption~\ref{exponent condition}}. This is different from the fractional Laplacian considered in \cite{FL24}. Indeed, \cite{FL24} requires additional assumptions on the odd dimensions to remove the resonance effect, which is only a technical reason.
	\end{remark}
	
	As an application of the above entanglement principle, we study the unique determination of a time-dependent potential $V$ in fractional poly-parabolic operators defined by
	\begin{equation}
		\label{equ: P_V def}
		P_V  :=\sum_{k=1}^N b_k\mathcal{H}_g^{s_k} +V ,
	\end{equation}
	where $V=V(x,t) \in L^\infty(\Omega_T)$. Here $0<s_1<\ldots<s_N<1$ and {$b_k>0$} for all $1\le k\le N$ (the positivity of $b_k$ is needed for the forward problem).  
	We consider the initial exterior value problem 
	\begin{equation}\label{equ: nonlocal para 2}
		\begin{cases}
			P_V  u  =0 &\text{ in }\Omega_T, \\
			u=f &\text{ in }(\Omega_e)_T, \\
			u=0 &\text{ in }\R^n \times \{t\leq -T\}.
		\end{cases}
	\end{equation}
	Assume that 
	$$
	\text{ $\{0\}$ is not a Dirichlet eigenvalue of $P_V$}.
	$$
	Theorem~\ref{thm: well-posedness} guarantees the well-posedness of the problem \eqref{equ: nonlocal para 2} and allows us to define the corresponding exterior DN map  
	$$
	\Lambda_V:\mathbf{H}^{s_N}((\Omega_e)_T) \to \mathbf{H}^{-s_N}((\Omega_e)_T), \quad f\mapsto \sum_{k=1}^N b_k\mathcal{H}_g^{s_k} u_f \bigg|_{(\Omega_e)_T} ,
	$$
	where $u_f$ is the unique solution of \eqref{equ: nonlocal para 2}. A rigorous definition of the DN map can be found in Section~\ref{sec: DN map}.
	
	We are interested in the following question. 
	\smallskip
	\begin{enumerate}[\textbf{(IP-2)}]
		\item \label{question: entanglement} 	\textit{Can one determine the potential $V\in L^\infty(\Omega_T)$ using the exterior DN map  
			$\Lambda_V$ of \eqref{equ: nonlocal para 2}?}
	\end{enumerate}
	\smallskip

	The second main result of the paper answers this question.  	
	\begin{theorem}[Global uniqueness]\label{thm: IP_ent} 
		Given $N\in \N$, $\{b_k\}_{k=1}^N\subset (0,\infty)$, and $0<s_1<\ldots<s_N<1$, let $\Omega\subset \R^n$ be a bounded Lipschitz domain for $n\geq 2$, and $W_1,\, W_2 \Subset \Omega_e$ be nonempty open subsets. Suppose $g \in C^\infty(\R^n;\R^{n\times n})$ satisfy \eqref{ellipticity}. Let $V_j=V_j(x,t) \in L^\infty(\Omega_T)$, and $\Lambda_{V_j}$ be the DN map of 
		\begin{equation}\label{equ: nonlocal para 3}
			\begin{cases}
				\big( \sum_{k=1}^N b_k\mathcal{H}_g^{s_k} + V_j \big)  u =0 &\text{ in }\Omega_T, \\
				u=f &\text{ in }(\Omega_e)_T, \\
				u=0 &\text{ in }\R^n \times \{t\leq -T\},
			\end{cases}
		\end{equation}
		for $j=1,2$. Then the relation
		\begin{equation}
			\begin{split}
				\Lambda_{V_1} f \big|_{(W_2)_T} = 	\Lambda_{V_2} f \big|_{(W_2)_T} , \quad \text{for any }f\in C^\infty_c((W_1)_T)
			\end{split}
		\end{equation}
		implies that $V_1=V_2$ in $\Omega_T$.
	\end{theorem}
	
	Note that Theorem~\ref{thm: IP_ent} extends the earlier results in \cite{LLR2019calder, BS2024calderon} (the works studied a global uniqueness for the case $N=1$) to multiple terms of nonlocal parabolic operators.

	\subsection{Organization of the article.}
	In Section~\ref{sec: prel}, we recall several functional spaces and introduce nonlocal parabolic operators through the semigroup theory, together with well-posedness results for initial exterior value problems and rigorous definitions of the DN maps. Section~\ref{sec: entanglement} is devoted to establishing the entanglement principle for nonlocal parabolic operators. Finally, in Section~\ref{sec: IP}, we prove the remaining main result regarding the global uniqueness of the potential in Theorem~\ref{thm: IP_ent}.

	
	\section{Preliminaries}\label{sec: prel}
	In this section, we introduce the function spaces used in this paper and recall several useful properties of the nonlocal parabolic operator $\mathcal{H}^s_g$.

	\subsection{Function spaces}\label{function space}
	We start by recalling the (fractional) Sobolev spaces. 
	Given $a\in \R$, $H^{a}(\R^{n})=W^{a,2}(\R^{n})$
	is the $L^2$-based fractional Sobolev space (see \cite{DNPV12} for example) with the norm 
	\[
	\|u\|_{H^{a}(\R^{n})}:=\left\Vert \mathcal{F}_x^{-1}\big\{\left\langle \xi\right\rangle ^{a}\mathcal{F}_xu\big\}\right\Vert _{L^{2}(\R^{n})},
	\]
	where $\left\langle \xi\right\rangle =(1+|\xi|^{2})^{\frac{1}{2}}$.
	Let $\mathcal O \subset \R^n$ be an open set. We define 
	\begin{align*}
		H^{a}(\mathcal{O}) & :=\{u|_{\mathcal{O}}:\,u\in H^{a}(\R^{n})\},\\
		\widetilde{H}^{a}(\mathcal{O}) & :=\text{closure of \ensuremath{C_{c}^{\infty}(\mathcal{O})} in \ensuremath{H^{a}(\R^{n})}}.
	\end{align*}
	The space $H^{a}(\mathcal{O})$ is complete under the norm
	\[
	\|u\|_{H^{a}(\mathcal{O})}:=\inf\left\{ \|v\|_{H^{a}(\R^{n})}:\ v\in H^{a}(\R^{n})\mbox{ and }v|_{\mathcal{O}}=u\right\} .
	\]

	Given an open set $B\subset \R^{n+1}$, if $f=f(x,t)$ and $g=g(x,t)$ are $L^2$ functions in $B$, we denote the $L^2$ inner product by 
	\begin{align*}
		( f,g )_B:=\int_{B} f\overline{g}\ dxdt.
	\end{align*}
	For the nonlocal space-time operator $\mathcal{H}^s=(\partial_t-\Delta)^s$, we will work on the following Lions-Magenes Sobolev spaces $H^{s,s/2}(\R^n\times \R)$ 
	(see \cite[Chapter 4.2]{LM12} and in particular equation (2.3) there). To simplify this notation and to emphasize the coupling between time and space variables, hereinafter we abbreviate it by $\mathbf{H}^s(\R^{n+1})$. More precisely, for $a\in \R$, we consider 
	\begin{align*}
		\mathbf{H}^a (\R^{n+1}):=\left\{u\in L^2(\R^{n+1}): \|u\|_{\mathbf{H}^a(\R^{n+1})}<\infty \right\} = H^{a,a/2}(\R^n\times \R),
	\end{align*}
	where 
	\begin{align}
		\label{eq:norm}
		\|u\|_{\mathbf{H}^{a}(\R^{n+1})}^2 = \int_{\R^{n+1}} (1+|i\rho+|\xi|^2|)^{a} |\widehat{u}(\xi,\rho)|^2 d\rho d\xi <\infty.
	\end{align}
	Note that $|i\rho +|\xi|^2|=\left(|\rho|^2+|\xi|^4\right)^{1/2}$ and $2^{-1/2}(|\rho|+|\xi|^2)\leq (|\rho|^2+|\xi|^4)^{1/2}\leq |\rho|+|\xi|^2$. 
	As the ``classical'' fractional Sobolev spaces (see \cite{ML-strongly-elliptic-systems}), the following notations follow naturally by treating spacetime together.  
	For an open set $O$ and a closed set $F$ in $\R^{n+1}$, $n\geq 1$, we define
	\begin{equation}
		\begin{split}
			\mathbf{H}^a(O) &:= \left\{u|_O:\ u\in \mathbf{H}^a(\R^{n+1})\right\},\\
			\widetilde{\mathbf{H}}^a(O) &:= \text{closure of $C^\infty_c(O)$ in $\mathbf{H}^a(\R^{n+1})$},\\ 
			\mathbf{H}^a_F= \mathbf{H}^a_F (\R^{n+1})&:=\left\{u\in \mathbf{H}^a(\R^{n+1}): \ \mathrm{supp}(u)\subset F \right\}.
		\end{split}
	\end{equation}
	Also, $C^\infty_c(\R^{n+1})$ is dense in $\mathbf{H}^a(\R^{n+1})$ under the norm $\|\cdot \|_{\mathbf{H}^a(\R^{n+1})}$. 
	Moreover,
	$$
	(\mathbf{H}^a(O))^* = \widetilde{\mathbf{H}}^{-a}(O),\quad  (\widetilde{\mathbf{H}}^a(O))^* = \mathbf{H}^{-a}(O),\text{ for } a\in\R.
	$$

	\subsection{The nonlocal parabolic operator}
	
	The definition for the nonlocal parabolic operator $\mathcal{H}^s_g$ was given in \cite{Balakrishnan_1960} (also see \cite{BDLCS2021harnack,BS2024calderon}) for $0<s<1$. It is known that heat operator $\p_t-\Delta_g$ in $\R^n\times \R$ possesses a globally defined fundamental solution $p_\tau (x,y)$, which satisfies    
	\[
	e^{\tau\Delta_g} 1(x)=\int_{\R^n} p_\tau(x,y)\,dV_g(y)=1, \text{ for every }x\in \R^n \text{ and }\tau>0,
	\]
	where $e^{\tau\Delta_g}$ stands for the heat semigroup associated to the operator $\Delta_g$.  
	For $u\in \mathcal{S}(\R^n)$, we have 
	\[
	e^{\tau\Delta_g} u(x)=\int_{\R^n} p_\tau(x,y)u(y)\, dV_g(y) , \text{ for every }x\in \R^n \text{ and }\tau>0. 
	\]
	Here the Riemannian volume form $dV_g$ is given by $dV_g(y)=\sqrt{|g|}\, dy$. For the purpose of simplifying the notation, in what follows, we will only use $dy$ to represent $dV_g(y)$.  Meanwhile, the heat kernels $p_\tau(x,y)$ satisfies 
	\begin{align}\label{est-heat-kernel}
		C_1 \LC \frac{1}{4\pi \tau}\RC^{n/2}e^{-\frac{c_1\abs{x-y}^2}{4\tau}} \leq p_\tau(x,y)\leq C_2 \LC \frac{1}{4\pi \tau}\RC^{n/2}e^{-\frac{c_2\abs{x-y}^2}{4\tau}} ,
	\end{align}
	for some positive constants $c_1,c_2,C_1$ and $C_2$ and for all $x,y\in \R^n$, $\tau>0$. 
	
	Since $\p_t$ and $-\Delta_g$ are commutable, we have $e^{-\tau \mathcal{H}_g}= e^{\tau \Delta_g} \circ e^{-\tau \p_t}$, and the evolution semigroup is given by 
	\begin{equation}\label{e-semni gp}
		\begin{split}
			e^{-\tau \mathcal{H}_g}u(x,t)&: = e^{\tau \Delta_g} u(x,t-\tau)=\int_{\R^n} p_\tau (x,y)u(y,t-\tau)\, dy, \quad \tau>0,\quad \text{ for }u\in \mathcal{S}(\R^{n+1}),
		\end{split}
	\end{equation}
	where $p_\tau(x,y)$ is the heat kernel given as before, and $\mathcal{S}(\R^{n+1})$ stands for the Schwartz space. Meanwhile, it is held that 
	\begin{equation}
		e^{-\tau \mathcal{H}_g}1(x,t)=\int_{\R^n}p_\tau (x,y)\, dy=1, \text{ for every }(x,t)\in \R^{n+1}, \text{ and }\tau>0.
	\end{equation}
	Note that $\left\{e^{-\tau \mathcal{H}_g}  \right\}_{\tau \geq 0}$ is a strongly continuous contractive semigroup such that\footnote{The notation $\mathcal{O}(\tau)$ is the Bachmann–Landau notation.}
	$$
	\norm{e^{-\tau \mathcal{H}_g} u-u}_{L^2(\R^{n+1})}=\mathcal{O}(\tau), \text{ as }\tau \to 0.
	$$
	
	Let us first give the explicit formula of $\mathcal{H}^s_g$ via the heat semigroup. 
	
	\begin{definition}[Balakrishnan formula]
		Given $s\in (0,1)$ and $u\in\mathcal{S}(\R^{n+1})$, the nonlocal parabolic operator $\mathcal{H}^s_g$ can be defined as (see \cite[Section 2]{BS2024calderon}) 
		\begin{align}\label{H^s}
			\mathcal{H}_g^s u(x,t):=\frac{1}{\Gamma(-s)} \int_0^\infty \LC e^{-\tau \mathcal{H}_g}u(x,t)-u(x,t)\RC \frac{d\tau}{\tau^{1+s}}.
		\end{align}
	\end{definition}
	
	Note that $\mathcal{H}^s_g u\in L^2(\R^{n+1})$ for $u\in \mathcal{S}(\R^{n+1})$ by the functional calculus. Using \eqref{e-semni gp},  
	we can rewrite \eqref{H^s} as follows:
	\begin{equation}\label{H^s explicit}
		\mathcal{H}_g^s u (x,t)=\int_0^\infty \int_{\R^n} \LC u(y,t-\tau)-u(x,t)\RC \mathcal{K}_s(x,y,\tau)\, dy d\tau,
	\end{equation}
	where 
	\begin{equation}
		\mathcal{K}_s(x,y,\tau):= \frac{1}{\Gamma(-s)}\frac{p_\tau(x,y)}{\tau^{1+s}}.
	\end{equation}
	In particular, as $g=\mathrm{I}_n$, the kernel $\mathcal{K}_s(x,y,\tau)$ has an explicit representation formula
	\begin{equation}\label{integral kernel for H^s}
		\mathcal{K}_s(x,y,\tau)= \frac{1}{(4\pi)^{n/2}\Gamma(-s)}\frac{e^{-\frac{\abs{x-y}^2}{4\tau}}}{\tau^{n/2+1+s}},
	\end{equation}
	where we used the heat kernel for the heat operator $\p_\tau -\Delta$ precisely.

	\begin{remark}
		Recall that the fractional Laplace--Beltrami operator can also be defined in a similar way, which is 
		\begin{equation}\label{fractional Laplace-Beltrami}
			\begin{split}
				\LC -\Delta_g \RC^s v(x)&:= \frac{1}{\Gamma(-s)}\int_0^\infty \LC e^{\tau \Delta_g} v(x)-v(x) \RC \frac{d\tau}{\tau^{1+s}} \\
				&\ = \int_0^\infty \int_{\R^n} (v(y)-v(x) ) \mathcal{K}_s(x,y,\tau)\, dy d\tau.
			\end{split}
		\end{equation}
	\end{remark}

	Using the Fourier transform with respect to the time variable $t\in \R$, one can express $\mathcal{H}_g^s u$ in terms of the Fourier transform. In doing so, we first denote by ${E_\lambda}$ the spectral measure associated to $\mathcal{H}_g$, i.e.,
	\[
	\mathcal{H}_g = - \int_0^\infty \lambda dE_\lambda.
	\]
	We then observe that the heat semigroup $\left\{e^{\tau \Delta_g}\right\}_{\tau\geq 0}$ can be written by spectral measures as an identity of gamma functions \cite[Section 2]{BS2024calderon}:
	\begin{align}\label{gamma function}
		e^{\tau \Delta_g}=\int_0^\infty e^{-\lambda \tau}\, dE_{\lambda} \quad \text{ and }\quad \frac{1}{\Gamma(-s)}\int_0^\infty \frac{e^{-(\lambda+\mathsf{i}\sigma)\tau}-1}{\tau^{1+s}}\, d\tau=(\lambda+\mathsf{i}\sigma)^s,
	\end{align}
	for $\lambda>0$ and $\sigma\in \R$, where $\mathsf{i}=\sqrt{-1}$. Taking the Fourier transform in the time variable on \eqref{e-semni gp} yields 
	\begin{equation}\label{Fourier e Hg}
		\mathcal{F}_t(e^{-\tau \mathcal{H}_g} u)(x,\sigma) =  e^{-i\sigma\tau }e^{ \tau \Delta_g} (\mathcal{F}_t u(\cdot,\sigma))(x).
	\end{equation}
	Together with \eqref{gamma function}, this gives the Fourier analogue of the definition \eqref{H^s} as follows:

	\begin{equation}\label{Fourier t}
		\mathcal{F}_t(\mathcal{H}_g^s u)(\cdot,\sigma) = \int_0^\infty (\lambda + i\sigma)^s \, dE_\lambda(\mathcal{F}_t u(\cdot,\sigma)).
	\end{equation}
	
	Moreover, we define the adjoint operator $\mathcal{H}_{g,\ast}^s$ of $\mathcal{H}_{g}^s$ in terms of the spectral resolution in the following manner
	\[
	\mathcal{F}_t (\mathcal{H}_{g,\ast}^s u)(\cdot, \sigma) = \int_0^\infty (\lambda - i\sigma)^s \, dE_\lambda(\mathcal{F}_t u(\cdot, \sigma)), \quad \text{for } u \in \mathcal{S}(\R^{n+1}).
	\]
	We also recall the following property from \cite[Section 2]{BS2024calderon}
	\begin{equation}\label{ajoint property}
		\begin{split}
			\langle \mathcal{H}_g^s f, h \rangle
			&= \big\langle \mathcal{H}_g^{s/2} f, \mathcal{H}_{g,\ast}^{s/2} h \big\rangle
			= \big\langle f, \mathcal{H}_{g,\ast}^s h \big\rangle\\
			&= \int_{\R} \int_0^\infty (\lambda + i\sigma)^s \, d\langle E_\lambda \mathcal{F}_t f, \overline{\mathcal{F}_t h} \rangle(\cdot, \sigma) \, d\sigma, \quad \text{for } f,\,h \in \mathcal{S}(\R^{n+1}).
		\end{split}
	\end{equation}
	From the resolution of the parabolic version of the Kato square root problem in \cite{AEN2020} and interpolation type argument, $\mathbf{H}^s(\R^{n+1})$ is the completion of $\mathcal{S}(\R^{n+1})$ with respect to the following norm:
	\begin{equation}
		\label{equ: norm equiv}
		\left( \int_{\mathbb{R}} \int_0^{\infty} \left( \left(1 + |\lambda + i\sigma|^2 \right)^{s/2} \, d\|E_\lambda(\mathcal{F}_t u(\cdot, \sigma))\|^2 \right) \, d\sigma \right)^{1/2}, \quad s\in (0,1),\quad \text{for } u\in \mathcal{S}(\R^{n+1}).
	\end{equation}
	Therefore, we get from the Cauchy-Schwarz inequality that 
	\begin{equation}\label{equ: C-S for H^s}
		\big\langle \mathcal{H}_g^s f, h \big\rangle= \big\langle \mathcal{H}_g^{s/2} f, \mathcal{H}_{g,\ast}^{s/2} h\big\rangle 
		\le C \|f\|_{\mathbf{H}^s(\R^{n+1})} \|h\|_{\mathbf{H}^s(\R^{n+1})}
	\end{equation}
	for some constant $C>0$ independent of $f,\,h$. This leads to the mapping properties 
	$$\mathcal{H}^s_g: \mathbf{H}^s(\R^{n+1})\to \mathbf{H}^{-s}(\R^{n+1})\quad\text{ and }\quad \mathcal{H}^s_{g,\ast}: \mathbf{H}^s(\R^{n+1})\to \mathbf{H}^{-s}(\R^{n+1}).$$
	We refer to \cite[Section 2]{BS2024calderon} for related discussions.
	
	For general $\alpha\in (0,\infty)\setminus \mathbb{N}$, we write $\alpha=m+s$, where $m$ is the integer part of $\alpha$ and $s\in (0,1)$. Based on \cite[Chapter 5]{CarracedoSanz2001}, we can write 
	\begin{equation}\label{H^alpha}
		\mathcal{H}^\alpha_g =\mathcal{H}_g^{m+s}=\mathcal{H}^{s}_g \big( \mathcal{H}^m_g \big) = \mathcal{H}^{m}_g \big( \mathcal{H}^s_g \big),
	\end{equation}
	where $\mathcal{H}^m_g=\LC \p_t -\Delta_g\RC^m$ is a local differential operator.

	\subsection{The well-posedness}	
	In this section, we will show the unique existence of solutions to the initial exterior value problems \eqref{equ: nonlocal para 2} by adapting the arguments developed in \cite{LLR2019calder, BS2024calderon}, in which the well-posed problems \eqref{equ: nonlocal para 1} were proved for a constant and a variable coefficient, respectively.

	To show the problem \eqref{equ: nonlocal para 2} has a unique solution, we consider the following initial exterior value problem instead:
	\begin{equation}\label{equ: wellposedness entangle}
		\begin{cases}
			P_V  u  =F &\text{ in }\Omega_T, \\
			u=f &\text{ in }(\Omega_e)_T, \\
			u=0 &\text{ in }\R^n \times \{t\leq -T\},
		\end{cases}
	\end{equation}
	where $F\in (\mathbf{H}^{s_N}_{\overline{\Omega_T}})^*$, $f\in \mathbf{H}^{s_N}((\Omega_e)_T)$, and $P_V$ is defined in \eqref{equ: P_V def}.
	
	Let $0<s_1<\ldots < s_N<1$, and consider the sesquilinear form $B_V(\cdot,\cdot)$ on $\mathbf{H}^{s_N}(\R^{n+1})\times \mathbf{H}^{s_N}(\R^{n+1})$ defined by 
	\[
	B_V(u,w) := \sum_{k=1}^N b_k \big( \mathcal{H}^{s_k/2}_g u,  \mathcal{H}^{s_k/2}_{g,\ast} w\big)_{L^2(\R^{n+1})} + \LC Vu, w \RC_{L^2(\Omega_T)}.
	\]
	According to \cite{LLR2019calder,cekic2020calderon}, we need to study a time-localized problem. We denote the cut-off of a function $u(x,t)$ on the time variable $t$ by 
	\[
	u_T(t,x) : = u(t,x)\chi_{[-T,T]}(t),
	\]
	where $\chi_{[-T,T]}(t)$ is a characteristic function for $t\in\R$. Since the characteristic function is a multiplier in the Sobolev space $H^\gamma(\R)$ for $|\gamma|<{1\over 2}$, we have $u_T\in \mathbf{H}^s(\R^{n+1})$ when $u\in \mathbf{H}^s(\R^{n+1})$ for $0<s<1$, see \cite[Section 2]{LLR2019calder}. We are ready to define the weak solution for \eqref{equ: wellposedness entangle}.

	\begin{definition}
		Let $\Omega$ be a bounded open set in $\R^n$ and $T>0$. Assume $0<s_1<\ldots < s_N<1$, and $V\in L^\infty(\Omega_T)$ such that $0$ is not a Dirichlet eigenvalue of the problem \eqref{equ: wellposedness entangle}. Given $F\in (\mathbf{H}^{s_N}_{\overline{\Omega_T}})^*$ and $f\in \widetilde{\mathbf{H}}^{s_N}((\Omega_e)_T)$, we say that $u\in \mathbf{H}^{s_N}(\R^{n+1})$ is a weak solution of \eqref{equ: wellposedness entangle} if $v: = (u - f)_T\in \mathbf{H}^{s_N}_{\overline{\Omega_T}}$ and 
		\[
		B_V(u,w) = \langle F,w \rangle_{(\mathbf{H}^{s_N}_{\overline{\Omega_T}})^*\times\mathbf{H}^{s_N}_{\overline{\Omega_T}}},\quad \text{for any } w\in \mathbf{H}^{s_N}_{\overline{\Omega_T}}.
		\]
	\end{definition}
	
	\begin{theorem}[Well-posedness]\label{thm: well-posedness}
		Let $\Omega$ be a bounded open set in $\R^n$ and $T>0$. Let $N\in \N$ and $\{b_k\}_{k=1}^N\subset (0,\infty)$. Suppose $g \in C^\infty(\R^n;\R^{n\times n})$ satisfies \eqref{ellipticity}. Assume $0<s_1<\ldots <s_N<1$ and $V\in L^\infty(\Omega_T)$ such that $0$ is not a Dirichlet eigenvalue of the problem \eqref{equ: wellposedness entangle}. Given $F\in (\mathbf{H}^{s_N}_{\overline{\Omega_T}})^*$ and $f\in \widetilde{\mathbf{H}}^{s_N}((\Omega_e)_T)$, there exists a unique solution $u_T\in \mathbf{H}^{s_N}(\R^{n+1})$ to the problem \eqref{equ: wellposedness entangle} satisfying 
		\[
		\norm{u_T}_{\mathbf{H}^{s_N}(\R^{n+1})} \le C\big(\|F\|_{(\mathbf{H}^{s_N}_{\overline{\Omega_T}})^*}+ \|f\|_{\mathbf{H}^{s_N}((\Omega_e)_T)}\big)
		\] 
		for some constant $C>0$ independent of $F$, $f$, and $u$.
	\end{theorem}
	\begin{proof}
		Let $v: = (u- f)_T$ and $\widetilde{F}: = F - P_Vf$, then $v\in \mathbf{H}^{s_N}_{\overline{\Omega_T}}$ and $v_T = v$. It suffices to show that for $\widetilde{F}\in(\mathbf{H}^{s_N}_{\overline{\Omega_T}})^*$, there exists a unique solution $v\in \mathbf{H}^{s_N}_{\overline{\Omega_T}}$ such that 
		\[
		B_V(v,w)  = \langle \widetilde{F}, w \rangle_{(\mathbf{H}^{s_N}_{\overline{\Omega_T}})^*\times\mathbf{H}^{s_N}_{\overline{\Omega_T}}},\quad \text{for any } w\in \mathbf{H}^{s_N}_{\overline{\Omega_T}}.
		\]

		Consider the bilinear form 
		\[
		B_V(v,w) + \mu (v,w)_{L^2(\Omega_T)}, 
		\]
		for $\mu \ge \|\min\{V,0\}\|_{L^\infty(\Omega_T)}$ in $\mathbf{H}^{s_N}_{\overline{\Omega_T}}$.
		The boundedness of this bilinear form
		\[
		\begin{split}
			B_V(v,w) + \mu (v,w)_{L^2(\Omega_T)} &\le C \sum_{k = 1}^N\|v\|_{\mathbf{H}^{s_k}(\R^{n+1})} \|w\|_{\mathbf{H}^{s_k}(\R^{n+1})} + C\|v\|_{L^2(\R^{n+1})}\|w\|_{L^2(\R^{n+1})}\\
			&\le C\|v\|_{\mathbf{H}^{s_N}(\R^{n+1})} \|w\|_{\mathbf{H}^{s_N}(\R^{n+1})} \quad \text{for any } v,w\in \mathbf{H}^{s_N}_{\overline{\Omega_T}},
		\end{split}
		\]
		follows directly from \eqref{equ: C-S for H^s}.
		
		We now prove the coercity in the space $\mathbf{H}^{s_N}_{\overline{\Omega_T}}$. Note that for $k = 1, \ldots, N$, we have for $v\in H^{s_N}_{\overline{\Omega_T}}$ that
		\begin{equation}\label{equ: H^{s_k/2}}
			\begin{split}
				\big\langle \mathcal{H}_g^{s_k/2} v, \mathcal{H}_{g,\ast}^{s_k/2} v \big\rangle 
				& =  \int_{\R} \int_0^\infty (\lambda + i\sigma)^{s_k} \, d\| E_\lambda (\mathcal{F}_t v) (\cdot, \sigma)\|^2 \, d\sigma \\
				& =  \int_{\R} \int_0^\infty |\lambda + i\sigma|^{s_k} \left(\cos(s_k\theta) + i \sin(s_k\theta))\right) \, d\| E_\lambda (\mathcal{F}_t v) (\cdot, \sigma)\|^2 \, d\sigma \\
				& = \int_{\R} \int_0^\infty |\lambda + i\sigma|^{s_k} \cos(s_k\theta)  \, d\| E_\lambda (\mathcal{F}_t v) (\cdot, \sigma)\|^2 \, d\sigma, \\
			\end{split}
		\end{equation}
		where $\tan \theta = \sigma/\lambda$ and we utilized the fact that $\sin(s_k\theta)$ is an odd function in the last step. Since $\lambda>0$ implies $\theta\in (-\frac{\pi}{2},\frac{\pi}{2})$, we have for all $s_k\in (0,1)$ that
		\[
		\cos(s_k\theta)\ge \cos\Big(\frac{s_k\pi}{2}\Big)\ge \min_{1\le k\le N} \cos\Big(\frac{s_k\pi}{2}\Big) =: c_s >0.
		\]
		Therefore, using the fact $\mu \ge \|\min\{V,0\}\|_{L^\infty(\Omega_T)}$, equation \eqref{equ: H^{s_k/2}} and the equivalent norm between \eqref{equ: norm equiv} and \eqref{eq:norm}, we obtain 
		\begin{equation}
			\label{equ: coercive-1}
			B_V(v,v)+ \mu (v,v)_{L^2(\Omega_T)}\ge \sum_{k=1}^N b_k \big( \mathcal{H}^{s_k/2}_g v,  \mathcal{H}^{s_k/2}_{g,\ast} v\big)_{\R^{n+1}} \ge C \sum_{k=1}^N \int_{\R^{n+1}} \left|i\rho+ |\xi|^2\right|^{s_k} |\widehat{v}(\xi,\rho)|^2 \,d\xi d\rho.  
		\end{equation}
		Applying the Hardy-Littlewood-Sobolev inequality for the $x$-variable and the fact that $v$ is compactly supported in $x$-variable, it yields for any $s_k\in (0,1)$,
		\begin{equation}\label{equ: Hardy-littlewood}
			\begin{split}
				\int_{\R^{n+1}} \left|i\rho+ |\xi|^2\right|^{s_k} |\widehat{v}(\xi,\rho)|^2 \,d\xi d\rho \ge \int_\R \|(-\Delta_x)^{s_k/2}\mathcal{F}_tv(\cdot,\rho)\|^2_{L^2(\R^n)}\,d\rho \ge C\|v\|^2_{L^2(\R^{n+1})},
			\end{split}
		\end{equation}
		see \cite{LLR2019calder,Banerjee2021Harnack} for a detailed explanation. Coercivity then follows from \eqref{equ: coercive-1} and \eqref{equ: Hardy-littlewood}, that is,  $B_V(v,v)+ \mu (v,v)_{L^2(\Omega_T)}\ge C\|v\|_{\mathbf{H}^{s_N}(\R^{n+1})}^2$.
		
		By Lax--Milgram theorem, there exists a unique solution $v = G_\mu\widetilde{F}\in H^{s_N}_{\overline{\Omega_T}}$ such that 
		\[
		B_V(v,w) + \mu(v,w)_{L^2(\Omega_T)}  = \langle \widetilde{F}, w \rangle_{(\mathbf{H}^{s_N}_{\overline{\Omega_T}})^*\times\mathbf{H}^{s_N}_{\overline{\Omega_T}}},\quad \text{for any} \quad w\in \mathbf{H}^{s_N}_{\overline{\Omega_T}},
		\]
		along with 
		\[
		\|v\|_{H^{s_N}_{\overline{\Omega_T}}} \le C \|\widetilde{F}\|_{(H^{s_N}_{\overline{\Omega_T}})^*}.
		\]
		In particular, $G_\mu : ( \mathcal{H}^{s_N}_{\Omega_T} )^* \to \mathcal{H}^{s_N}_{\Omega_T}$ is bounded and by the compact Sobolev embedding, the operator $G_\mu: L^2(\Omega_T) \to L^2(\Omega_T)$ is compact. Then the spectral theorem implies that the eigenvalues of $G_\mu$ are $\frac{1}{\lambda_j + \mu}$ with $\lambda_j \to +\infty$. Fredholm alternative and the assumption $0$ is not a Dirichlet eigenvalue of $P_V$ ensure the existence and uniqueness of the problem under consideration.
	\end{proof}
	
	\begin{remark}\label{remark of well-posedness}
		The assumption that $\{0\}$ is not a Dirichlet eigenvalue of $P_V$ implies $\{0\}$ is not a Dirichlet eigenvalue of the adjoint of $P_V$.
		Similarly, we can establish the well-posedness result for the adjoint problem to \eqref{equ: wellposedness entangle}. Under the hypothesis of Theorem \ref{thm: well-posedness}, there exists a unique solution $u\in\mathbf{H}^{s_N}(\R^{n+1})$ to the future exterior problem 
		\[
		\begin{cases}
			\big( \sum_{k=1}^N b_k \mathcal{H}_{g,\ast}^{s_k} + V \big)  u =F &\text{ in }\Omega_T, \\
			u=f &\text{ in }(\Omega_e)_T, \\
			u=0 &\text{ in }\R^n \times \{t\ge T\}.
		\end{cases}    
		\]
	\end{remark}

	\subsection{The Dirichlet-to-Neumann map}\label{sec: DN map}
	Based on the well-posedness results of initial exterior problems \eqref{equ: nonlocal para 2}, let us define the corresponding DN maps $\Lambda_V$ by means of the bilinear form $B_V$. 
	
	We first introduce the following quotient spaces for our exterior data by
	\[
	\mathbb{X} : = \mathbf{H}^{s_N}(\R^n\times[-T,T])/\mathbf{H}^{s_N}_{\overline{\Omega_T}}, 
	\]
	equipped with the norm 
	\begin{equation}
		\begin{split}
			\|[f]\|_{\mathbb{X}}&: = \inf_{\phi\in \mathbf{H}^{s_N}_{\overline{\Omega_T}}} \|f+\phi\|_{H^s(\R^n)},\qquad \ \ \,   \text{ for } f\in \mathbf{H}^{s_N}(\R^n\times[-T,T]).
		\end{split}
	\end{equation}
	Denote $\mathbb{X}^*$ as the dual of $\mathbb{X}$. 
	We now define the DN maps as follows:
	\begin{align*}
		\big\langle\Lambda_V [f], [\zeta]\big\rangle_{\mathbb{X}^*\times \mathbb{X}} &: = B_V(u_f, \zeta), &\text{ for } [f],[\zeta]\in \mathbb{X},
	\end{align*}
	where $u_f\in \mathbf{H}^{s_N}(\R^{n+1})$ is the solution of \eqref{equ: nonlocal para 2} with the Dirichlet data $f$.
	
	Analogously, one can also define the adjoint DN maps by utilizing the following natural pairing property
	\begin{align*}
		\big\langle [f], \Lambda_V^*[\zeta]\big\rangle_{\mathbb{X}\times \mathbb{X}^*} &:= \big\langle\Lambda_V [f], [\zeta]\big\rangle_{\mathbb{X}^*\times \mathbb{X}}, &\text{ for } [f],[\zeta]\in \mathbb{X},
	\end{align*}
	Also, the adjoint DN maps can be represented as 
	$$
	\langle [f], \Lambda_V^*[\zeta]\rangle_{\mathbb{X}_1\times \mathbb{X}_1^*} =  B_V(f, u_\zeta),\quad \langle [h], 
	$$
	where $u_\zeta \in \mathbf{H}^{s_N}(\R^{n+1})$ is the solution of the adjoint equation $\big( \sum_{k=1}^N b_k \mathcal{H}_{g,\ast}^{s_k} + V \big)  u_\zeta =0$ with the Dirichlet data $\zeta$ in $(\Omega_e)_T$ and $u_\zeta=0$ for $t\geq T$.
	To simplify the notations, we use $f$ to denote $[f]$. 
	\begin{proposition}\label{prop: DN}
		Let $\Omega$ be a bounded open set in $\R^n$ and $T>0$. Let $N\in \mathbb{N}$. Assume that $\{b_k\}_{k=1}^N\subset (0,\infty)$, $0<s_1<\ldots <s_N<1$, and $V\in L^\infty(\Omega_T)$ such that $0$ is not a Dirichlet eigenvalue of the problem \eqref{equ: wellposedness entangle}. Then the DN map $\Lambda_V$ defined above is well-defined and bounded. 
	\end{proposition}
	\begin{proof}
		We first show that $\Lambda_V$ only depends on the equivalence classes. For $f,\,\zeta\in \mathbf{H}^{s_N}(\R^n\times[-T,T])$. Let $\phi,\,\psi\in \mathbf{H}^{s_N}_{\overline{\Omega_T}}$. Since $u_f$ and $u_{f+\phi}$ both solve the equation \eqref{equ: nonlocal para 2} with the same exterior data, Theorem~\ref{thm: well-posedness} implies $u_f = u_{f+\phi}$.
		By the linearity of $B_V$ in the second component, it yields
		\[
		B_V(u_{f+\phi},\zeta+\psi) = B_V(u_f, \zeta+\psi) = B_V(u_f, \zeta)+ B_V(u_f, \psi).
		\]
		Using the fact that $\supp(\psi) \subset \overline{\Omega_T}$ and $u_f$ solves \eqref{equ: nonlocal para 2}, we get $B_V(u_f,\psi) = 0$. This proves that $\langle \Lambda_V (f+\phi), (\zeta+\psi) \rangle =\langle \Lambda_Vf,\zeta \rangle $ and thus $\Lambda_V$ is well-defined.
		
		The boundedness of $\Lambda_V$ follows from 
		\[
		|\langle\Lambda_Vf, \zeta\rangle |\le |B_{V}(u_{f+\phi},\zeta+\psi) | \le C\|u_{f+\phi}\|_{H^{s_N}(\R^{n+1})}\|\zeta+\psi\|_{H^{s_N}(\R^{n+1})} 
		\]
		and taking the infimum with respect to $\phi,\,\psi\in \mathbf{H}^{s_N}_{\overline{\Omega_T}}$.
	\end{proof}

	\section{The entanglement principle}\label{sec: entanglement}	
	The aim of this section is to show the entanglement principles for the fractional parabolic operators on the Euclidean domain. To this end, we first recall the result demonstrated in \cite[Proposition 3.1]{FKU24}, which will play a crucial role of decoupling the mixed fractional parabolic operators later.

	\begin{proposition}[\text{\cite[Proposition 3.1]{FKU24}}]\label{prop: key of entangle}
		Let $N\in \N$ and $\{\alpha_k \}_{k=1}^N \subset (0,\infty) \setminus \N$ satisfy \textbf{Assumption~\ref{exponent condition}}. Given $a>0$, suppose that $\{f_k \}_{k=1}^N\subset C^\infty((0,\infty))$,  there exist positive constants $c$ and $\delta $ such that the function $f=f_k$ ($k=1,\ldots, N$) fulfills 
		\begin{equation}\label{key exponential decay condition}
			\begin{split}
				|f(\tau)|\leq ce^{-\delta \tau }, \quad \tau \in (a,\infty) , \quad \text{and}\quad |f(\tau)|\leq ce^{-\frac{\delta}{\tau}}, \quad \tau \in (0,a].
			\end{split}
		\end{equation}
		Additionally, if there exists $\ell \in \N\cup \{0\}$ such that 
		\begin{equation}\label{FKU key id}
			\begin{split}
				\sum_{k=1}^N \Gamma(m+1+\alpha_k) \int_0^\infty f_k (\tau) \tau^{-m}\, d\tau =0 , \quad \text{for all}\quad m=\ell, \,\ell+1,\,\ell +2 , \ldots,
			\end{split}
		\end{equation}
		then $f_k(\tau)=0$ for all $\tau \in (0,\infty)$, and for all $k=1,\ldots, N$. 
	\end{proposition}

	Note that \cite[Proposition 3.1]{FKU24} shows the case when $a=1$. The same result still holds for any given constant $a>0$ by following the same arguments there.

	\begin{remark}
		Let us emphasize the essential difference of the entanglement principle between \cite{FKU24, FL24} and this work. 
		\begin{enumerate}[(i)]
			\item 	The works \cite{FKU24, FL24} investigate the entanglement principle for nonlocal elliptic operators, and we study an analogous tenet for the nonlocal parabolic operator. Particularly, in \cite{FKU24}, the authors investigated the entanglement principle for fractional Laplace--Beltrami operators on closed Riemannian manifolds. Thanks to the compactness, the first inequality in \eqref{key exponential decay condition} can be achieved naturally by its heat kernel estimate. 
			
			\item In \cite{FL24}, the authors considered the same problem on $\R^n$, for the fractional Laplace operator. Due to the lack of compactness, the first inequality in \eqref{key exponential decay condition} can not be satisfied. Hence, the authors introduced the super-exponential decay condition, allowing them to transfer the problem to the spherical mean vanishing property. 
			In this work, we are in the non-compact setting as well as \cite{FL24}. However, thanks to the representation formulas \eqref{e-semni gp} and \eqref{H^s}, one can introduce a suitable decay condition of $u(x,t)$ with respect to the time variable, so that the first inequality in \eqref{key exponential decay condition} still holds. Given this, one may expect the entanglement principle to hold for the nonlocal parabolic operator.
			
			\item Let us point out that the entanglement principle for fractional Laplace--Beltrami operators $(-\Delta_g)^s$ remains open in the non-compact Euclidean space $\R^n$, which seems to be a challenging problem to resolve.
			
		\end{enumerate}

	\end{remark}

	The following theorem lays the foundation of the proof of Theorem \ref{thm: ent} for $\{\alpha_k \}_{k=1}^N \subset (0,\infty)\setminus\mathbb{N}$. 
	
	\begin{theorem}\label{thm_ent_smooth}
		Let $\mathcal{O}\subset \R^n$ be a nonempty open set for $n\geq 2$. Let $N\in\mathbb{N}$, $T>0$, and $0<s_1 <\ldots <s_N<1$. Suppose that $\{ v_k\}_{k=1}^{N}\subset C^\infty((-\infty,T); \mathcal{S}(\R^n))$ satisfies the following estimates: given any multi-index $\beta=(\beta_0,\beta_1,\ldots,\beta_n) \in \LC \N \cup \{0\} \RC^{n+1}$, there exist positive constants $C_\beta$ and $\delta$ such that
		\begin{equation}\label{solution exponential decay time v_k}
			\begin{split}
				\big|D^\beta_{x,t} v_k(x,t)\big|\leq \begin{cases}
					C_0 \big| \varphi_0(x)\big| e^{ \delta t},& |\beta|=0\\
					C_\beta|\varphi_\beta(x)|,& |\beta|\ge 1
				\end{cases},  \quad \text{for }(x,t)\in \R^n \times \{t\leq-T\},
			\end{split}
		\end{equation}
		and $k=1,\ldots, N$, where 
		$\varphi_\beta \in \mathcal{S}(\R^n)$. If
		\begin{equation}\label{condition_ent_2}
			\left. v_1 \right|_{\mathcal{O}_T} =\ldots = 	\left. v_N\right|_{\mathcal{O}_T} =0 \quad\text{ and }\quad   \bigg(\sum_{k=1}^N \mathcal{H}_g^{s_k} v_k\bigg)\bigg|_{\mathcal{O}_T} =0  
		\end{equation}
		hold, then $v_k\equiv 0$ in $\R^n_T$ for all $k=1,\ldots, N$.
	\end{theorem}
	
	\begin{proof}
		Similar to the arguments as in \cite{FKU24,FL24,LLU2022para}, via the condition \eqref{condition_ent_2}, the iteration arguments yield for $m = 1,2,\ldots$ that 
		\begin{equation}
			\mathcal{H}^{m}_gv_1  \big|_{\mathcal{O}_T}=\ldots = \mathcal{H}^{m}_gv_N  \big|_{\mathcal{O}_T}=0, \quad\text{ and }\quad \bigg( \sum_{k = 1}^N \mathcal{H}_g^{s_k} \mathcal{H}_g^m v_k \bigg) \bigg|_{\mathcal{O}_T} = 0. 
		\end{equation}
		Let $\omega\Subset\mathcal{O}$ be an open nonempty subset such that 
		\begin{equation}\label{distance condition}
			\dist (\omega , \R^n\setminus \overline{\mathcal{O}}) \geq 2\kappa ,
		\end{equation}
		for some constant $\kappa>0$.
		For $(x,t)\in \omega_T$, we have by \eqref{H^s} that 
		\begin{align}\label{ID:entangle}
			0 =\sum_{k = 1}^N \mathcal{H}_g^{s_k} \mathcal{H}_g^m v_k (x,t) & = \sum_{k = 1}^N \frac{1}{\Gamma(-s_k)} \int_0^\infty  \LC e^{-\tau \mathcal{H}_g} \mathcal{H}_g^mv_k\RC(x,t)  \frac{d\tau}{\tau^{1+s_k}} \\
			& = \underbrace{\sum_{k = 1}^N \frac{1}{\Gamma(-s_k)} \int_0^\infty  (-1)^m \p_\tau^m \LC e^{-\tau \mathcal{H}_g} v_k\RC (x,t)  \frac{d\tau}{\tau^{1+s_k}}}_{(e^{-\tau \mathcal{H}_g} \mathcal{H}_g^mv_k)(x,t) = (-1)^m \p_\tau^m (e^{-\tau \mathcal{H}_g} v_k)(x,t)}.
		\end{align}
		
		Next, fix $t_0\in (-T,T)$, we shall show no contribution arises at the endpoints when conducting integration by parts in $\tau$. That is, for $\ell = 0,\,1,\ldots, \, m-1$, the following terms 
		\begin{equation}
			\label{eq_boundaries}
			\p_\tau^\ell \left(e^{-\tau \mathcal{H}_g}v_k\right)(x,t_0)\frac{1}{\tau^{s_k+m-\ell}}
		\end{equation}
		vanish at $\tau \to 0^+$ and $\tau \to  +\infty$. 
		Since $t_0-(-T)>0$, let us denote 
		$$
		a:= t_0-(-T)=t_0+T>0.
		$$
		Under this assumption, we have 
		\begin{equation}\label{tau intervals 1}
			t_0 -\tau \in (-\infty,-T) \quad \text{if }\tau \in (a,\infty),
		\end{equation}
		and
		\begin{equation}\label{tau intervals 2}
			t_0-\tau \in [-T,T)  \quad \text{if }\tau \in (0,a].
		\end{equation}
		To show the boundary terms \eqref{eq_boundaries} vanish, it suffices to show the following two estimates. For $\tau\in(a,\infty)$, we have
		\begin{equation}\label{eq_boundary_infty}
			\begin{split}
				\big| \p_\tau^\ell\big( e^{-\tau \mathcal{H}_g} v_k \big) (x,t_0) \big|
				&= \big|\left(e^{-\tau\mathcal{H}_g}\mathcal{H}_g^\ell v_k\right)(x,t_0)\big|\\
				& \leq \underbrace{C_2 \int_{\R^n} \LC \frac{1}{4\pi \tau}\RC^{n/2}e^{-\frac{c_2|x-y|^2}{4\tau}} \big| \mathcal{H}_g^\ell v_k(y,t_0-\tau)\big|\, dy}_{\text{By \eqref{est-heat-kernel}}}\\ 
				&\leq \underbrace{\frac{C}{\tau^{n/2}} \int_{\R^n} e^{-\frac{c_2|x-y|^2}{4\tau}}\big| \varphi_\ell(y)\big|\, dy}_{\text{By \eqref{tau intervals 1} and \eqref{solution exponential decay time v_k} and $\mathcal{H}_g^\ell = (\p_t-\Delta_g)^\ell$}}\\
				&\leq \underbrace{C\big\|\varphi_\ell\big\|_{L^\infty(\R^n)} }_{\int_{\R^n} e^{-\frac{|x|^2}{4\tau}}\,dx = (4 \pi\tau)^{\frac{n}{2}}
				},
			\end{split}
		\end{equation}
		for some $\varphi_\ell\in \mathcal{S}(\R^n)$,
		and thus
		\[
		\big|\p_\tau^\ell \left(e^{-\tau \mathcal{H}_g}v_k\right)(x,t_0)\frac{1}{\tau^{s_k+m-\ell}}\big|\le C\norm{\varphi_\ell}_{L^\infty(\R^n)}\frac{1}{\tau^{s_k+m-\ell}} \to 0,\quad \text{as}\quad \tau\to \infty.
		\]
		Similarly, for $\tau\in(0,a]$, we have  
		\begin{equation}\label{eq_boundary_0}
			\begin{split}
				\big|\p_\tau^\ell \big( e^{-\tau \mathcal{H}_g} v_k \big) (x,t_0) \big| &\leq C_2 \int_{\R^n} \LC \frac{1}{4\pi \tau}\RC^{n/2}e^{-\frac{c_2|x-y|^2}{4\tau}} \big| \mathcal{H}_g^\ell v_k(y,t_0-\tau)\big|\, dy \\
				&\leq \underbrace{\frac{C}{\tau^{n/2}}\int_{\R^n\setminus \overline{\mathcal{O}}} e^{-\frac{c_2|x-y|^2}{4\tau}} \big| \mathcal{H}_g^\ell v_k(y,t_0-\tau)\big|\, dy}_{\text{By \eqref{tau intervals 2} so that $\mathcal{H}_g^\ell v_k(y,t_0-\tau)=0$ for $y\in \mathcal{O}$}} \\
				&\leq \frac{C}{\tau^{n/2}}e^{-\frac{c_2\kappa^2}{\tau}}\sup_{t\in (-T,T)}\left\| \mathcal{H}_g^\ell v_k(\cdot,t) \right\|_{L^1(\R^n\setminus \overline{\mathcal{O}})} \\
				&\leq C e^{-\frac{c}{\tau}},\quad \text{ for }x\in \omega\Subset\mathcal{O},
			\end{split}
		\end{equation}
		which leads to
		\[
		\big|\p_\tau^\ell \left(e^{-\tau \mathcal{H}_g}v_k\right)(x,t_0)\frac{1}{\tau^{s_k+m-\ell}}\big|\le Ce^{-\frac{c}{\tau}}\frac{1}{\tau^{s_k+m-\ell}} \to 0,\quad \text{as}\quad \tau\to 0^+.
		\]
		
		Since we have shown that the boundary values vanish, by applying $m$-times integration by parts in $\tau$ to \eqref{ID:entangle}, we get  
		\begin{align}\label{comp ent 1}
			0&=\sum_{k=1}^N \frac{\gamma_k}{\Gamma(-s_k)}\int_0^\infty \big( e^{-\tau \mathcal{H}_g}v_k\big)(x,t_0) \tau^{-(m+1+s_k)}\,  d\tau\\
			&=\sum_{k=1}^N \Gamma(m+1+s_k)\int_0^\infty f_k(\tau) \tau^{-m}\, d\tau,\quad \text{for any $m\in \N\cup \{0\}$},
		\end{align}
		where  the function
		\begin{equation}\label{comp ent 3}
			f_k(\tau):=\frac{1}{\Gamma(-s_k) \Gamma(1+s_k)}\big( e^{-\tau \mathcal{H}_g}  v_k \big) (x,t_0) \tau^{-(1+s_k)},
		\end{equation}
		and the constant $\gamma_k$ is defined as $\gamma_k:= (1+s_k)(2+s_k)\ldots (m+s_k)= \frac{\Gamma(m+1+s_k)}{\Gamma(1+s_k)}$. Note that the smoothness of $v_k(x,t)$ and $p_\tau(x,y)$ yield $f_k(\tau) \in C^\infty((0,\infty))$.
		
		In addition, we will show that $f_k(\tau)$ satisfies the bound \eqref{key exponential decay condition} so that we will be able to apply Proposition~\ref{prop: key of entangle}. As \eqref{eq_boundary_0} for $\ell = 0$ implies the case $\tau\in (0,a]$, it remains to show the exponential decay in the interval $\tau\in (a,\infty)$. 		
		To this end, by following a similar argument as in \eqref{eq_boundary_infty} and utilizing \eqref{solution exponential decay time v_k} with $|\beta|=0$, it gives rise to the succeeding estimate for $f_k(\tau)$, $\tau\in (a,\infty)$:
		\begin{equation}\label{eq_tau_infty}
			\begin{split}
				\big|f_k(\tau)|
				&\le \big|\left(e^{-\tau\mathcal{H}_g} v_k\right)(x,t_0) \tau^{-(1+s_k)}\big|\\
				& \leq \underbrace{C_2\tau^{-(1+s_k)} \int_{\R^n} \LC \frac{1}{4\pi \tau}\RC^{n/2}e^{-\frac{c_2|x-y|^2}{4\tau}} \big|v_k(y,t_0-\tau)\big|\, dy}_{\text{By \eqref{est-heat-kernel}}}\\ 
				&\leq \underbrace{\frac{C\tau^{-(1+s_k)}}{\tau^{n/2}} e^{\delta(t_0-\tau)} \int_{\R^n} e^{-\frac{c_2|x-y|^2}{4\tau}}\big| \varphi_0(y)\big|\, dy}_{\text{By \eqref{tau intervals 1} and \eqref{solution exponential decay time v_k} }} \\ 
				&\leq \underbrace{C\tau^{-(1+s_k)}e^{\delta(t_0-\tau)}  \big\|\varphi_0\big\|_{L^\infty(\R^n)} }_{\int_{\R^n} e^{-\frac{|x|^2}{4\tau}}\,dx = (4 \pi\tau)^{\frac{n}{2}}
				}\\
				&\leq \underbrace{Ce^{\delta T}e^{-\delta\tau}\norm{\varphi_0}_{L^\infty(\R^n)}}_{\text{since }t_0\in (-T,T)},
			\end{split}
		\end{equation}
		for some $\varphi_0\in \mathcal{S}(\R^n)$.
		With this estimate, we can now apply Proposition \ref{prop: key of entangle} to obtain that $f_k(\tau)$ is identically zero, for $\tau \in (0,\infty)$. Indeed the definition of $f_k(\tau)$ in \eqref{comp ent 3} implies
		\begin{equation}
			\begin{split}
				\big( e^{-\tau \mathcal{H}_g}v_k \big) (x,t_0)=0 \quad \text{for }x\in \omega, \ \tau>0 \text{ and }k=1,\ldots, N.
			\end{split}
		\end{equation}
		Since $t_0\in (-T,T)$ can be arbitrary, we further deduce
		\begin{equation}\label{comp ent 7}
			\begin{split}
				\big( e^{-\tau \mathcal{H}_g}v_k \big) (x,t)=0 \quad \text{for }(x,t)\in \omega_T, \ \tau>0 \text{ and }k=1,\ldots, N.
			\end{split}
		\end{equation}
		Now, since $\omega\Subset \mathcal{O}$ and $\kappa>0$ in \eqref{distance condition} are arbitrarily chosen, as a result, substituting \eqref{comp ent 7} and \eqref{condition_ent_2} into \eqref{H^s}, we have 
		\begin{equation}
			v_k =\mathcal{H}_g^{s_k} v_k=0, \text{ in } \mathcal{O}_T, \text{ for }k=1,\ldots, N.
		\end{equation}	
		Finally, applying the (weak) UCP for nonlocal parabolic operators $\mathcal{H}^s_g$, $s\in (0,1)$ (see \cite{BS2024calderon,LLR2019calder}), we can ensure $v_k\equiv 0 $ in $\R^n_T$, for $k=1,\ldots, N$. This concludes the proof.     
	\end{proof}	
	
	With Theorem \ref{thm_ent_smooth}, we can prove Theorem \ref{thm: ent}.  
	
	\begin{proof}[Proof of Theorem \ref{thm: ent}]
		For $\alpha_k=m_k+s_k\in \R_+\setminus \mathbb{N}$, where $m_k$ is the integer part of $\alpha_k$ and $s_k\in (0,1)$ is the fractional part of $\alpha_k$. In particular, we have $\mathcal{H}^{\alpha_k}_g =\mathcal{H}_g^{m_k+s_k}=\mathcal{H}^{s_k}_g \big( \mathcal{H}^{m_k}_g \big)$. 
		Note that $\{u_k\}_{k=1}^N \subset C^\infty((-\infty,T); \mathcal{S}(\R^n))$, and we let
		$$
		v_k:= b_k\mathcal{H}^{m_k}_g u_k,\quad b_k\in \mathbb{C}\setminus\{0\}, \quad\text{ for } k=1,\ldots, N,
		$$ 
		which is also in $C^\infty((-\infty,T); \mathcal{S}(\R^n))$. Since $\mathcal{H}^{m_k}_g$ is a local operator, the condition \eqref{condition_entanglement_u} implies \eqref{condition_ent_2}, and the condition \eqref{solution exponential decay time} leads to \eqref{solution exponential decay time v_k}. By Theorem \ref{thm_ent_smooth}, we deduce $v_k=0$ in $\R^n_T$, and thus $\mathcal{H}^{m_k}_gu_k=0$ in $\R^n_T$ with $u_k|_{\mathcal{O}_T}=0$. Lastly, the UCP of the classical parabolic operators $\mathcal{H}^{m_k}_g$ in $\R^n_T$ leads to the desired result $u_k=0$ in $\R^n_T$. This proves the assertion.
	\end{proof}
	We conclude this section with some remarks regarding the exponential decay conditions. 
	\begin{remark}~\label{remark:entangle section 3}
		\begin{enumerate}[(i)]
			\item             
			The exponential decay condition \eqref{solution exponential decay time v_k} is needed only for $N\geq 2$ (entangled nonlocal parabolic) in Theorem~\ref{thm_ent_smooth} in order to break the nonlocal effect arising from every fractional operator $\mathcal{H}^{s_k}_g$. When $N=1$ (single $\mathcal{H}^s_g$), however, the UCP holds without such decay condition and it has been shown in the works \cite{BS2024calderon,LLR2019calder,LLU2022para}.
			As a result, \eqref{solution exponential decay time} in Theorem~\ref{thm: ent} can be removed when one considers $N=1$.

			\item Although the decay conditions \eqref{solution exponential decay time} and \eqref{solution exponential decay time v_k} seem strong, for the study of inverse problems, both \eqref{solution exponential decay time} and \eqref{solution exponential decay time v_k} hold automatically provided that the solution of the initial exterior value problem has zero initial data, namely, $u=0$ in $\R^n\times \{t\leq -T\}$, see Section~\ref{sec: IP} for detailed discussions.
		\end{enumerate}
	\end{remark}

	\section{Inverse problems and proof of main results}\label{sec: IP}
	
	\subsection{Global uniqueness for the fractional poly-parabolic operators}
	
	It is known that the proof of uniqueness can be established by employing the UCP together with the Runge approximation property. In what follows, we present an alternative formulation of the entanglement principle by additionally imposing an initial value vanishing condition \eqref{condition in prop ent 2}, which helps to shorten the arguments for the UCP in \cite{LLR2019calder,BS2024calderon}. It is worth mentioning that when we deal with the inverse problem, \eqref{condition in prop ent 2} is fulfilled naturally due to the initial condition in the problem under consideration.  
	
	\begin{proposition}[Modified entanglement principle]\label{prop: entangle in IP}
		Let $\mathcal{O}\subset \R^n$ be a nonempty open set for $n\geq 2$. Let $N\in\mathbb{N}$, $T>0$, $\{b_k\}_{k=1}^N\subset (0,\infty)$, and $0<s_1 <\ldots <s_N<1$. Let $u_k\in \mathbf{H}^{s_k}(\R^{n+1})$, for $k=1,\ldots, N$.
		If 
		\begin{equation}\label{condition in prop ent 2}
			u_1 =\ldots =u_N=0 \text{ in }\R^n\times \{t\leq -T\},
		\end{equation}
		and
		\begin{equation}\label{condition in prop ent 1}
			u_1|_{\mathcal{O}_T}=\ldots = 	u_N|_{\mathcal{O}_T} =\bigg(  \sum_{k=1}^N b_k\mathcal{H}_g^{s_k} u_k\bigg) \bigg|_{\mathcal{O}_T}=0,
		\end{equation} 
		hold, then $u_1=\ldots = u_N=0$ in $\R^n_T$ for all $k = 1,\ldots,N$.
	\end{proposition}
	
	\begin{proof} 
		The proof can be reduced to Proposition~\ref{prop: key of entangle} after an appropriate deduction. 
		Note that the functions $u_k \in \mathbf{H}^{s_k}(\R^{n+1})$ are not necessarily smooth, thus Theorem~\ref{thm_ent_smooth} cannot be applied directly to the current setting. 
		To address this lack of smoothness, we utilize certain properties of the heat kernel, combined with a smooth mollifier, which enables us to approximate the functions $u_k$ and thereby overcome the regularity issue.

		To this end, for $\varepsilon>0$, we denote 
		$$
		T_\varepsilon:=T-
		\varepsilon.
		$$
		Consider the one-dimensional standard mollifier $\varphi\in C^\infty_0(\R)$ with compact support $\supp\varphi \subset (-1,1)$, and satisfy $0\leq \varphi$ and $\|\varphi\|_{L^2(\R)}=1$. For each $\varepsilon>0$, we define
		$\varphi_\varepsilon(t) :=\eps^{-1}\varphi(t/\varepsilon)$ and thus $\varphi_\varepsilon\in C^\infty(\R)$ with  $\supp\varphi_\varepsilon \subset (- \varepsilon, \varepsilon)$. For each $x\in \R^n$, since $u_k(x,\cdot)$ is locally integrable in $t$ variable, the function
		\begin{equation}
			u_{k,\varepsilon}(x,t):= ( u_k\ast \varphi_\varepsilon ) (x,t) =\int_{-\varepsilon}^\varepsilon u_k(x,t-\eta) \varphi_\varepsilon(\eta)\, d\eta, \quad t\in (-\infty,T_\varepsilon),
		\end{equation}
		and  $u_{k,\varepsilon}(x,\cdot)\in C^\infty(\R)$. Also, $u_{k,\varepsilon}(x,\cdot)\rightarrow u_k(x,\cdot)$ almost everywhere as $\varepsilon\rightarrow 0$ for $k=1,\ldots, N$. 
		
		Next, for $s\in (0,1)$, recalling the definition \eqref{H^s} and using a direct computation give
		\begin{equation}
			\big( \mathcal{H}^s_gu_{k,\varepsilon}\big)(x,t) =	\big( \mathcal{H}^s_g( u_k\ast \varphi_\varepsilon)\big) (x,t) = \big( (\mathcal{H}^s_g u_k)\ast \varphi_{\varepsilon}  \big) (x,t),\quad(x,t)\in \mathcal{O}_{T_\varepsilon}, 
		\end{equation}
		which can be seen since $\mathcal{H}^s_gu(x,t)$ is defined via a convolution in $t$. This implies
		$$
		\sum_{k=1}^N ((\mathcal{H}_g^{s_k}u_k)\ast \varphi_\varepsilon) (x,t) = \sum_{k=1}^N (\mathcal{H}_g^{s_k}u_{k,\varepsilon})(x,t),\quad (x,t)\in \mathcal{O}_{T_\varepsilon}.
		$$
		Together with \eqref{condition in prop ent 1} and \eqref{condition in prop ent 2}, 
		we get
		\begin{equation}\label{ID:u mollifier_1}
			\begin{alignedat}{2}
				&u_{1,\varepsilon}=\ldots =  u_{N,\varepsilon} =0, &&\text{ in } \left(\mathcal{O}\times (-T_\varepsilon-2\varepsilon,T_\varepsilon)\right)\cup\left(\R^n\times \{t\leq -T_{\varepsilon}-2\varepsilon\}\right),\\
				&\sum_{k=1}^N \mathcal{H}_g^{s_k}u_{k,\varepsilon}=0, &&\text{ in } \mathcal{O}_{T_\varepsilon}.
			\end{alignedat}
		\end{equation}
		Applying $\mathcal{H}^m_g$, $m=1,\,2,\ldots,$ to \eqref{ID:u mollifier_1} leads to 
		\begin{equation}
			\mathcal{H}^m_g u_{1,\varepsilon}|_{\mathcal{O}_{T_\varepsilon}}=\ldots = 	\mathcal{H}^m_g	 u_{N,\varepsilon}|_{\mathcal{O}_{T_\varepsilon}} =\LC\sum_{k=1}^N b_k\mathcal{H}_g^{m+s_k}u_{k,\varepsilon}\RC\bigg|_{\mathcal{O}_{T_\varepsilon}}=0,
		\end{equation}

		Let $\omega\Subset\mathcal{O}$ such that 
		\begin{equation}\label{distance condition_Prop4.1}
			\dist (\omega , \R^n\setminus \overline{\mathcal{O}}) \geq 2\kappa ,
		\end{equation}
		for some constant $\kappa>0$.
		For $(x,t)\in \omega_{T_\varepsilon}$, we have by \eqref{H^s} that  
		\begin{align}\label{ID:entangle_Prop4.1}
			0 =\sum_{k = 1}^N {b_k}\mathcal{H}_g^{s_k} \mathcal{H}_g^m u_{k,\varepsilon} (x,t) & = \sum_{k = 1}^N \frac{{b_k}}{\Gamma(-s_k)} \int_0^\infty  \LC e^{-\tau \mathcal{H}_g} \mathcal{H}_g^m u_{k,\varepsilon}\RC(x,t)  \frac{d\tau}{\tau^{1+s_k}} \\
			& = \sum_{k = 1}^N \frac{{b_k}}{\Gamma(-s_k)} \int_0^\infty  (-1)^m \p_\tau^m \LC e^{-\tau \mathcal{H}_g} u_{k,\varepsilon}\RC(x,t)  \frac{d\tau}{\tau^{1+s_k}}.
		\end{align}

		Next, fix $t_0\in (-T_\varepsilon,T_\varepsilon)$, we shall show no contribution arises at the endpoints when conducting integration by parts in $\tau$. That is, for $\ell = 0,\,1,\ldots, \, m-1$, the following terms 
		\begin{equation}
			\label{eq_boundaries_Prop4.1}
			\p_\tau^\ell \left(e^{-\tau \mathcal{H}_g} u_{k,\varepsilon}\right)(x,t_0)\frac{1}{\tau^{s_k+m-\ell}}
		\end{equation}
		vanish at $\tau \to 0^+$ and $\tau \to  +\infty$. 
		
		Let us denote 
		$$
		a:= t_0-(-T)=t_0+T>0,
		$$
		and split $\tau$ into the following regions:  
		\begin{equation}\label{tau intervals 1-H^s}
			t_0 -\tau \in (-\infty,-T_\varepsilon-2\varepsilon) \quad \text{if }\tau \in (a+\varepsilon,\infty),
		\end{equation}
		and
		\begin{equation}\label{tau intervals 2-H^s}
			t_0-\tau \in [-T_\varepsilon-2\varepsilon,T_\varepsilon)  \quad \text{if }\tau \in (0,a+\varepsilon].
		\end{equation}
		We first show the boundary terms \eqref{eq_boundaries_Prop4.1} vanish when $\tau\to \infty$. For $\tau\in(a+\varepsilon,\infty)$, we have $t_0 -\tau \in (-\infty,-T_\varepsilon-2\varepsilon)$, and  
		\[
		u_{1,\varepsilon}=\ldots =  u_{N,\varepsilon}=0 \text{ in } \R^n\times \{t\leq -T_{\varepsilon}-2\varepsilon\}.
		\]
		Therefore        
		\begin{equation}\label{eq_boundary_infty_Prop4.1}
			\begin{split}
				\big| \p_\tau^\ell\big( e^{-\tau \mathcal{H}_g} u_{k,\varepsilon} \big) (x,t_0) \big| 
				&= \big|\left(e^{-\tau\mathcal{H}_g}\mathcal{H}_g^\ell u_{k,\varepsilon}\right)(x,t_0)\big|\\
				& \leq C_2 \int_{\R^n} \LC \frac{1}{4\pi \tau}\RC^{n/2}e^{-\frac{c_2|x-y|^2}{4\tau}} \big| \mathcal{H}_g^\ell u_{k,\varepsilon}(y,t_0-\tau)\big|\, dy\\
				&=0.
			\end{split}
		\end{equation}

		Now, notice the function $\big( e^{-\tau \mathcal{H}_g} u_{k,\varepsilon} \big) (x,t)$ is $C^\infty$-smooth for $(x,t,{\tau}) \in \R^{n+1}\times (0,\infty)$. This can be seen via the integral formula \eqref{e-semni gp} and the heat kernel $p_\tau(x,y)$ is $C^\infty$-smooth for $(x,y)\in \R^n\times \R^n$ and $\tau>0$ (see \cite{HeatKernelsSpectralTheory}, Chapter 5). Moreover, it is also known that the function $\big( e^{-\tau \mathcal{H}_g} u_{k,\varepsilon} \big) (x,t)$ satisfies 
		\begin{equation}\label{extension heat equation}
			\begin{cases}
				(\p_\tau + \mathcal{H}_g ) \big( e^{-\tau \mathcal{H}_g} u_{k,\varepsilon} \big)(x,t) =0 &\text{ for }(x,t,\tau)\in \R^{n+1}\times \R_+, \\
				\displaystyle \lim_{\tau \to 0^+}\big( e^{-\tau \mathcal{H}_g} u_{k,\varepsilon} \big) (x,t) = u_{k,\varepsilon}(x,t) &\text{ for }(x,t)\in \R^{n+1}, 
			\end{cases}
		\end{equation}
		where the above limit holds in the $L^2$-sense, for all $k=1,\ldots, N$.

		Next, we shall show no contribution arises at $\tau \to 0^+$. 
		For $\tau\in(0,a+\varepsilon]$, we have $t_0-\tau \in [-T_\varepsilon-2\varepsilon,T_\varepsilon) $ and 
		\[
		u_{1,\varepsilon}=\ldots =  u_{N,\varepsilon} =0 \text{ in } \mathcal{O}\times (-T_\varepsilon-2\varepsilon,T_\varepsilon).
		\]
		For $x\in \omega \Subset \mathcal{O}$, using \eqref{extension heat equation} and binomial expansion, the Lebesgue dominated convergence theorem infers that 
		\begin{equation}\label{eq_boundary_0_new1}
			\begin{split}
				\big| \p_\tau^\ell \big( e^{-\tau \mathcal{H}_g} u_{k,\varepsilon} \big) (x,t_0)\big| &=\big|(\p_t -\Delta_g)^\ell\big( e^{-\tau \mathcal{H}_g} u_{k,\varepsilon} \big) (x,t_0)\big| \\
				&=\bigg|  \int_{\R^n}\sum_{i=1}^\ell \left( \begin{matrix}
					\ell \\
					i \end{matrix} \right) \p_t^i (-\Delta_g)^{\ell-i} \big( p_\tau(x,y)u_{k,\varepsilon}\big)(y,t_0-\tau)\, dy \bigg| \\
				& =\bigg| \sum_{i=1}^\ell \left( \begin{matrix}
					\ell \\
					i  \end{matrix} \right) \int_{\R^n} (-\Delta_g)^{\ell-i}p_\tau(x,y)\p_t^i u_{k,\varepsilon}(y,t_0-\tau)\, dy  \bigg|\\
				&=\bigg| \sum_{i=1}^\ell \left( \begin{matrix}
					\ell \\
					i  \end{matrix} \right) \int_{\R^n\setminus \mathcal{O}} (-\Delta_g)^{\ell-i}p_\tau(x,y)\p_t^i u_{k,\varepsilon}(y,t_0-\tau)\, dy  \bigg|\\
				&=\bigg| \sum_{i=1}^\ell \left( \begin{matrix}
					\ell \\
					i  \end{matrix} \right) \int_{\R^n\setminus \mathcal{O}} (-\p_\tau)^{\ell-i}p_\tau(x,y)\p_t^i u_{k,\varepsilon}(y,t_0-\tau)\, dy  \bigg|,
			\end{split}
		\end{equation}
		where we used $p_\tau(x,y)$ is the heat kernel solving the heat equation 
		$$
		(\p_\tau -\Delta_g )p_\tau(x,y)=0,
		$$
		for $x\neq y$, $x,y\in \R^n$ and $\tau>0$ (since $x\in \omega$ and $y\in \R^n\setminus \mathcal{O}$).
		Utilizing \cite[Theorem 3.1]{Grigoryan_upper_heat_kernel}, it is known that the heat kernel satisfies the following time-derivative estimate  
		\begin{equation}\label{time-derivative heat est}
			\begin{split}
				\big|\p_\tau^{\ell-i}p_\tau (x,y)\big| \leq C \frac{(1+\abs{x-y}^2/\tau)^{N'}}{\tau^{\ell-i}\min(\tau, R^2)^l}e^{-\frac{\abs{x-y}^2}{4\tau}},
			\end{split}
		\end{equation}
		for any $\ell\in \N$, $R>0$, $x,y\in \R^n$ and for some constants $C,N',l>0$ {with $N'=\ell-i+l+1$.}
		Thus, inserting \eqref{time-derivative heat est} into \eqref{eq_boundary_0_new1}, we have
		\begin{equation}\label{eq_boundary_0_new2}
			\begin{split}
				\big| \p_\tau^\ell \big( e^{-\tau \mathcal{H}_g} u_{k,\varepsilon} \big) (x,t_0)\big|\leq C \sum_{i=1}^\ell \left( \begin{matrix}
					\ell \\
					i \end{matrix} \right) \int_{\R^n\setminus \mathcal{O}} \frac{(1+\abs{x-y}^2/\tau)^{N'}}{\tau^{\ell-i}\min(\tau, R^2)^l}e^{-\frac{\abs{x-y}^2}{4\tau}} \big| \p_t^i u_{k,\varepsilon}(y,t_0-\tau)\big| \, dy.
			\end{split}
		\end{equation}
		
		Thanks to the bound \eqref{eq_boundary_0_new2}, for $\tau\in(0,a+\varepsilon]$, by the H\"older inequality, we have 
		\begin{equation}\label{eq_boundary_0_new3}
			\begin{split}
				&\quad \,  \int_{\R^n\setminus \mathcal{O}} \frac{(1+\abs{x-y}^2/\tau)^{N'}}{\tau^{\ell-i}\min(\tau, R^2)^l}e^{-\frac{\abs{x-y}^2}{4\tau}} \big| \p_t^i u_{k,\varepsilon}(y,t_0-\tau)\big| \, dy \\
				&\leq Ce^{-\frac{c_0}{\tau}}\sup_{t_0-\tau\in (-T_\varepsilon-2\varepsilon,T_\varepsilon)}\big\| \p_t^i u_{k,\varepsilon}(\cdot,t_0-\tau)\big\|_{L^2(\R^n)}\underbrace{\bigg(\int_{\kappa} ^\infty \frac{(1+\abs{\rho}^2/\tau)^{2N'}}{\tau^{2\ell-2i}\min(\tau, R^2)^{2l}}e^{-\frac{c_1\abs{\rho}^2}{\tau}}\rho^{n-1}\, d\rho \bigg)^{1/2}}_{\text{change of variable }\rho/\sqrt{\tau}\mapsto \rho}\\
				& \le Ce^{-\frac{c_0}{\tau}}\sup_{t_0-\tau\in (-T_\varepsilon-2\varepsilon,T_\varepsilon)}
				\big\| \p_t^i u_{k,\varepsilon}(\cdot,t_0-\tau)\big\|_{L^2(\R^n)} \Bigg(\frac{\tau^{n/2}}{\tau^{2\ell-2i}\min(\tau,R^2)^{2l}}
				\underbrace{ \int_0^\infty (1+\rho^2)^{2N'}\rho^{n-1} e^{-c_1\rho^2}\,d\rho }_{\text{finite}}\Bigg)^{1/2}, 
			\end{split}
		\end{equation}
		and, moreover, the Minkowski inequality and Young's inequality give the following bound
		$$
		\big\| \p_t^i u_{k,\varepsilon}(\cdot,t_0-\tau)\big\|_{L^2(\R^n)}\leq \|u\|_{L^2(\R^{n+1})}\|\p^i_t\varphi_\varepsilon\|_{L^2(\R)}\leq \|u\|_{\mathbf{H}^{s_k}(\R^{n+1})}\|\p^i_t\varphi_\varepsilon\|_{L^2(\R)}
		$$
		for some constants $c_0, c_1,C>0$. Recalling that $u_{k,\eps}(x,t)$ is smooth in $t$, using the estimate \eqref{eq_boundary_0_new3}, one can ensure
		\[
		\big|\p_\tau^\ell \left(e^{-\tau \mathcal{H}_g}u_{k,\varepsilon}\right)(x,t_0)\frac{1}{\tau^{s_k+m-\ell}}\big|\to 0,\quad \text{as}\quad \tau\to 0^+.
		\]

		Since we have shown that the boundary values vanish, by applying $m$-times integration by parts in $\tau$ to \eqref{ID:entangle_Prop4.1}, we get  
		\begin{align}\label{comp ent 1_1}
			0&=\sum_{k=1}^N \frac{\gamma_k}{\Gamma(-s_k)}\int_0^\infty \big( e^{-\tau \mathcal{H}_g}u_{k,\varepsilon}\big)(x,t_0) \tau^{-(m+1+s_k)}\,  d\tau\\
			&=\sum_{k=1}^N \Gamma(m+1+s_k)\int_0^\infty f_{k,\varepsilon}(\tau) \tau^{-m}\, d\tau,\quad \text{for any $m\in \N\cup \{0\}$},
		\end{align}
		where for each fixed $\varepsilon>0$, the function $f_{k,\varepsilon}(\tau)$ is defined as
		\begin{equation}\label{comp ent 3_Prop4.1}
			f_{k,\varepsilon}(\tau):=\frac{{b_k}}{\Gamma(-s_k) \Gamma(1+s_k)}\big( e^{-\tau \mathcal{H}_g}  u_{k,\varepsilon} \big) (x,t_0) \tau^{-(1+s_k)},
		\end{equation}
		and the constant $\gamma_k$ is defined as
		\begin{equation}\label{comp ent 2}
			\gamma_k:=  (1+s_k)(2+s_k)\ldots (m+s_k)= \frac{\Gamma(m+1+s_k)}{\Gamma(1+s_k)}.   
		\end{equation}
		Note that the smoothness of $u_{k,\varepsilon}(x,t)$ in time variable brings out $f_{k,\varepsilon}(\tau) \in C^\infty((0,\infty))$.
		
		In addition, we will show that $f_{k,\varepsilon}(\tau)$ satisfies the bound \eqref{key exponential decay condition} so that we will be able to apply Proposition~\ref{prop: key of entangle}. 
		\begin{itemize}
			\item 	For $\tau\in (a+\varepsilon,\infty)$, $f_{k,\varepsilon}(\tau) = 0$ is implied by \eqref{eq_boundary_infty_Prop4.1} by taking $\ell = 0$.
			
			\item For $\tau\in (0,a+\varepsilon]$, we can derive similarly as in \eqref{eq_boundary_0_new3} by taking $\ell=i=0$ and get
			\[
			|f_{k,\varepsilon}(\tau)|\le Ce^{-\frac{c_2}{\tau}} ,
			\]
			for some constant $c_2>0$.
			
		\end{itemize}
		With this estimate, we can now apply Proposition \ref{prop: key of entangle} to obtain that $f_{k,\varepsilon}(\tau)$ is identically zero, for $\tau \in (0,\infty)$ and for all $k = 1,\ldots, N$. Indeed the definition of $f_{k,\varepsilon}(\tau)$ in \eqref{comp ent 3_Prop4.1} implies
		\begin{equation}
			\begin{split}
				\big( e^{-\tau \mathcal{H}_g}u_{k,\varepsilon} \big) (x,t_0)=0 \quad \text{for }x\in \omega, \ \tau>0 \text{ and }k=1,\ldots, N.
			\end{split}
		\end{equation}
		Since $t_0\in (-T_\varepsilon,T_\varepsilon)$ can be arbitrary, we further deduce
		\begin{equation}\label{comp ent 7_Prop4.1}
			\begin{split}
				\big( e^{-\tau \mathcal{H}_g}u_{k,\varepsilon} \big) (x,t)=0 \quad \text{for }(x,t)\in \omega_{T_\varepsilon}, \ \tau>0 \text{ and }k=1,\ldots, N.
			\end{split}
		\end{equation}
		Now, since $\omega\Subset \mathcal{O}$ and $\kappa>0$ are arbitrarily chosen, we further have 
		\begin{equation}
			u_{k,\varepsilon} =\mathcal{H}_g^{s_k} u_{k,\varepsilon}=0, \text{ in } \mathcal{O}_{T_\varepsilon}, \text{ for }k=1,\ldots, N.
		\end{equation}	
		Applying the (weak) UCP for nonlocal parabolic operators $\mathcal{H}^{s_k}_g$, $s_k\in (0,1)$ (see \cite{BS2024calderon,LLR2019calder}), we can ensure $u_{k,\varepsilon}\equiv 0 $ in $\R^n_{T_\varepsilon}$, for $k=1,\ldots, N$. Then $u_k \equiv 0$ in $\R^n_T$ follows from the fact that $u_{k,\varepsilon}$ converges to $u_k$ almost everywhere as $\varepsilon\to 0$.
		This proves the assertion.
	\end{proof}

	\begin{remark}
		Let us emphasize that the nonlocal operator $\mathcal{H}^s_g$ has constant coefficients in the time variable. This allows the use of a convolution argument to relax the regularity assumptions for certain functions. Consequently, one may weaken the regularity hypotheses in Theorem~\ref{thm: ent}.
	\end{remark}
	
	\begin{remark}\label{remark for entangle}
		The entanglement principle also applies to the adjoint fractional poly-parabolic operator $\sum_{k=1}^N b_k \mathcal{H}^{s_k}_{g,\ast}$. More precisely, if 
		\begin{equation*}
			u_1 =\ldots =u_N=0 \text{ in }\R^n\times \{t\ge T\},
		\end{equation*}
		and
		\begin{equation*}
			u_1|_{\mathcal{O}_T}=\ldots = 	u_N|_{\mathcal{O}_T} = \bigg( \sum_{k=1}^N b_k\mathcal{H}_{g,\ast}^{s_k}u_k\bigg)\bigg|_{\mathcal{O}_T}=0,
		\end{equation*}
		then $u_k \equiv 0$ in $\R^n_T$ for all $k=1,\ldots,N$. The proof proceeds in the same way as that of Proposition~\ref{prop: entangle in IP}, except that one reverses the sign in the $t$-variable.
	\end{remark}
	
	To study the inverse problems, we only need one single function in the entanglement principle to prove our result, i.e., $u:=u_1=\ldots =u_N$ in $\R^{n+1}$. Below, we will apply Proposition~\ref{prop: entangle in IP} to prove the Runge approximation for fractional poly-parabolic operators. 
	
	For $0<s_1<\ldots <s_N<1$, $\{b_k\}_{k=1}^N\subset (0,\infty)$, and $T>0$, we recall the notation $\Omega_T= \Omega\times (-T,T) \subset \R^{n+1}$. Let $V\in L^\infty(\Omega_T)$ satisfy the eigenvalue condition \eqref{eigenvalue condition}. For $f\in \widetilde{\mathbf{H}}^{s_N}((\Omega_e)_T)$, let $u_f\in \mathbf{H}^{s_N}(\R^{n+1})$ solve the problem
	\begin{align}\label{equ in Runge}
		\begin{cases}
			\Big(\sum_{k=1}^N b_k\mathcal{H}_g^{s_k}+V\Big) u_f=0 &\text{ in }\Omega_T, \\
			u_f=f &\text{ in } (\Omega_e)_T, \\
			u_f=0 &\text{ in }\R^n \times \{t\leq -T\}. 
		\end{cases}
	\end{align}
	It is known that $\chi_{(-\infty,T]}(t)u_f(t,x)$ is the unique solution of \eqref{equ in Runge}.

	\begin{lemma}[Runge approximation] \label{lemma: Runge approximation}
		For $n\geq 1$, let $W \subset \Omega_e$ be a nonempty open subset and $T>0$ be a real number. Then the set 
		$$
		\mathcal{R}=\left\{u_f|_{\Omega_T }:\, u_f\ \text{ is the solution to \eqref{equ in Runge}},\ f\in C^\infty_c(W_T) \right\}
		$$
		is dense in $L^2(\Omega_T)$.
	\end{lemma}

	\begin{proof}
		The proof is standard and relies on the Hahn-Banach theorem. It suffices to show that if $(v,w)_{L^2(\Omega_T)}=0$ for all $v\in \mathcal{R}$, then necessarily $w\equiv 0$.  
		To proceed, let $w\in L^2(\Omega_T)$. Assume that
		\begin{align*}
			\left(\chi_{(-\infty,T]}u_f,w\right)_{L^2(\Omega_T)}
			= \left(u_f,w\right)_{L^2(\Omega_T)}=0, \quad \text{for all } f\in C_c^\infty(W_T),
		\end{align*}
		where $\chi_{(-\infty,T]}u_f$ denotes the unique solution of \eqref{equ in Runge} in $\Omega_T$.  
		Here we have used the fact that, as before, the future data does not influence the solution in $\Omega_T$.
		
		Next, let $\phi\in \mathbf{H}^s(\R^{n+1})$ be the solution of 
		\begin{align}\label{equation of phi}
			\begin{cases}
				\Big( \sum_{k=1}^N b_k\mathcal{H}_{g,\ast}^{s_k} + V\Big) \phi=w & \text{in }  \Omega_T,\\
				\phi =0 & \text{in } (\Omega_{e})_T \cup (\R^n \times (\R\setminus (-T,T))),
			\end{cases}
		\end{align}
		where the well-posedness of \eqref{equation of phi} is guaranteed by Remark \ref{remark of well-posedness}.
		Then,
		\begin{align}\label{Runge:phi}
			0=(u_f, w)_{L^2(\Omega_T)} = \Big(u_f -f,  \Big( \sum_{k=1}^N b_k\mathcal{H}_{g,\ast}^{s_k} + V\Big) \phi\Big)_{L^2(\R^n_T)} = -\Big(f,  \sum_{k=1}^N b_k\mathcal{H}_{g,\ast}^{s_k}\phi \Big)_{L^2(W_T) },
		\end{align}	
		for all $f\in C^\infty_c(W_T)$, where in the last identity we used the fact that $f$ is supported in $W_T$ and $u_f$ solves \eqref{equ in Runge}. 
		As \eqref{Runge:phi} holds for all $f\in C^\infty_c(W_T)$, it yields
		$$
		\sum_{k=1}^N b_k\mathcal{H}_{g,\ast}^{s_k} \phi = 0\text{ in } W_T.
		$$
		Combining it with $\phi= 0$ in $W_T$ (from \eqref{equation of phi}),
		we apply the entanglement principle (see Remark \ref{remark for entangle}) to deduce
		$$
		\phi=0 \text{ in }\R^n_T.
		$$
		Moreover, from \eqref{equation of phi} again, the exterior condition of $\phi$ in the past and future time vanish, which implies $\phi\equiv 0 $ in $\R^{n+1}$. Hence we infer that $\mathcal{H}^{s_k}_{g,\ast} \phi=0$ in $\R^{n+1}$, for all $k=1,\ldots, N$. Finally, by substituting this $\phi$ back into \eqref{equation of phi}, we can conclude $w \equiv 0$, which proves the Runge approximation.
	\end{proof}
	
	Before proving Theorem \ref{thm: IP_ent}, we also need the following integral identity.
	
	\begin{lemma}[Integral identity]\label{Lem Integral identitiy}
		Let $\Omega_T\subset \R^{n+1}$ be the bounded open set and let $V_1,\,V_2\in L^\infty(\Omega_T)$ satisfy the eigenvalue condition \eqref{eigenvalue condition}. Then, for any exterior Dirichlet data $f_1,\,f_2\in \widetilde{\mathbf{H}}^{s_N}((\Omega_e)_T)$, we have 
		\begin{align}\label{eq:integral identity}
			\left\langle (\Lambda_{V_1}-\Lambda_{V_2})f_1,f_2\right\rangle_{\mathbf{H}^s((\Omega_e)_T)^\ast \times \mathbf{H}^s((\Omega_e)_T)}=\left((V_1-V_2)u_1,u_2\right)_{\Omega_T},
		\end{align}
		where $u_1\in \mathbf{H}^{s_N}(\R^{n+1})$ is the weak solution of
		\begin{equation}\label{equ u_1 in integral id}
			\begin{cases}
				\big( \sum_{k=1}^N b_k\mathcal{H}^{s_k}_g +V_1 \big) u_1=0 &\text{ in }\Omega_T ,\\
				u_1 = f_1 &\text{ in }(\Omega_e)_T,\\
				u_1=0 &\text{ in }\R^n\times \{t\leq -T\},
			\end{cases}
		\end{equation}
		and $u_2\in \mathbf{H}^{s_N}(\R^{n+1})$ is the weak solution of 
		\begin{equation}\label{equ u_2 in integral id}
			\begin{cases}
				\big( \sum_{k=1}^N b_k \mathcal{H}^{s_k}_{g,\ast} +V_2 \big) u_2=0 &\text{ in }\Omega_T, \\
				u_2=f_2 &\text{ in }(\Omega_e)_T,\\
				u_2=0 &\text{ in }\R^n \times \{t\geq T\}.
			\end{cases}
		\end{equation}
	\end{lemma}
	
	\begin{proof}
		By the adjoint property, the DN map, one has 
		\begin{equation}
			\begin{split}
				&\quad \, \left\langle (\Lambda_{V_1}-\Lambda_{V_2})f_1,f_2\right\rangle_{\mathbf{H}^s((\Omega_e)_T)^\ast \times \mathbf{H}^s((\Omega_e)_T)}\\
				&=\left\langle \Lambda_{V_1}f_1,f_2\right\rangle_{\mathbf{H}^s((\Omega_e)_T)^\ast \times \mathbf{H}^s((\Omega_e)_T)}-\left\langle f_1,\Lambda_{V_2}^*f_2\right\rangle_{_{\mathbf{H}^s((\Omega_e)_T) \times \mathbf{H}^s((\Omega_e)_T)^\ast}} \\
				&=B_{V_1}(u_1,u_2)-B_{V_2}(u_1,u_2)\\
				&=\left((V_1-V_2)u_1|_{\Omega_T},u_2|_{\Omega_T}\right)_{\Omega_T}. 
			\end{split}
		\end{equation}
		This completes the proof.
	\end{proof}

	Now, we can prove Theorem \ref{thm: IP_ent}.
	
	\begin{proof}[Proof of Theorem \ref{thm: IP_ent}]
		We follow the same argument as the proof of \cite[Theorem 1.1]{LLR2019calder}. 	
		If $\left. \Lambda_{V_{1}}f\right|_{(W_{2})_T}=\left.\Lambda_{V_{2}}f\right|_{(W_{2})_T}$ for any $f\in C_{c}^{\infty}((W_{1})_T)$, where $W_{1}$ and $W_{2}$ are nonempty open subsets of $\Omega_{e}$. By the integral identity \eqref{eq:integral identity}, we have 
		\[
		\int_{\Omega_T}(V_{1}-V_{2})u_{1}\,\overline{u_{2}}\, dxdt=0,
		\]
		where $u_{1}, \,u_2\in \mathbf H^{s_N}(\R^{n})$ solve $\big( \sum_{k=1}^N b_k \mathcal{H}^{s_k}_{g } +V_1 \big) u_1=0 $ and $\big( \sum_{k=1}^N b_k \mathcal{H}^{s_k}_{g,\ast} +V_2 \big) u_2=0$ with $u_1=0$ for $\{t\leq -T\}$ and $u_2=0$ for $\{t\geq T\}$. Also, $u_{1}$, $u_2$ have the same exterior value $f \in C_{c}^{\infty}((W_{1})_T)$. 
		
		Given an arbitrary $\phi\in L^{2}(\Omega_T)$ and by using the Runge approximation of Lemma \ref{lemma: Runge approximation}, there exists two sequences of functions $\{ u_{\ell}^{1}\}_{\ell\in \N}$, $\{u_{\ell}^{2}\}_{\ell\in \N}\subset  \mathbf H^{s_N}(\mathbb{R}^{n+1})$ that fulfill
		\begin{align*}
			&\big( \sum_{k=1}^N b_k \mathcal{H}^{s_k}_{g } +V_1 \big) u_{\ell}^{1}=\big( \sum_{k=1}^N b_k \mathcal{H}^{s_k}_{g,\ast} +V_2 \big) u_{\ell}^{2}=0\text{ in }\Omega_T,\\
			& \supp\big(u_{\ell}^{1}\big)\subseteq\overline{(\Omega_{1})_T}, \quad \supp\big(u_{\ell}^{2}\big)\subseteq \overline{(\Omega_{2})_T},\\
			& \left. u_{\ell}^{1}\right|_{\Omega_T}=\phi+r_{\ell}^{1},\quad \left. u_{\ell}^{2}\right|_{\Omega_T}=1+r_{\ell}^{2},
		\end{align*}
		where $\Omega_{1}$, $\Omega_{2}\subset \R^n$ are two open sets containing $\Omega$, and $r_{\ell}^{1},\,r_{\ell}^{2}\to 0$ in $L^{2}(\Omega_T)$ as $\ell\to\infty$. By substituting the solutions $u^j_\ell$ into the integral identity and passing to the limit as $\ell\to\infty$, we infer that 
		\[	
		\int_{\Omega_T}\LC V_{1}-V_{2}\RC \phi \, dxdt=0.
		\]
		As $\phi\in L^{2}(\Omega_T)$ is arbitrary, we can conclude that $V_{1}=V_{2}$ in $\Omega_T$. This completes the proof.
	\end{proof}

	\section*{Statements and Declarations}
	
	\noindent\textbf{Data availability statement.} 
	No datasets were generated or analyzed during the current study.
	
	\medskip 
	
	\noindent\textbf{Conflict of Interests.}  Hereby, we declare there are no conflicts of interest.

	\medskip 
	
	\noindent\textbf{Acknowledgment.} 
	\begin{itemize}
		\item R.-Y. Lai is partially supported by the National Science Foundation through the grant DMS-2306221. 
		\item Y.-H. Lin is partially supported by the Ministry of Science and Technology, Taiwan, under projects 113-2628-M-A49-003 and 113-2115-M-A49-017-MY3. Y.-H. Lin is also a Humboldt research fellow for experienced researchers from Germany.
	\end{itemize}
	
	\bibliography{refs} 

\begin{thebibliography}{BDLCRS21}

\bibitem[AEN20]{AEN2020}
Pascal Auscher, Moritz Egert, and Kaj Nystr\"{o}m.
\newblock {$\rm L^2$} well-posedness of boundary value problems for parabolic
  systems with measurable coefficients.
\newblock {\em J. Eur. Math. Soc. (JEMS)}, 22(9):2943--3058, 2020.

\bibitem[Bal60]{Balakrishnan_1960}
Alampallam~V. Balakrishnan.
\newblock Fractional powers of closed operators and the semigroups generated by
  them.
\newblock {\em Pacific J. Math.}, 10:419--437, 1960.

\bibitem[BDLCRS21]{BDLCS2021harnack}
Animesh Biswas, Marta De~Le\'on-Contreras, and Pablo Ra\'ul~Stinga.
\newblock Harnack inequalities and {H}\"older estimates for master equations.
\newblock {\em SIAM J. Math. Anal.}, 53(2):2319--2348, 2021.

\bibitem[BGMN21]{Banerjee2021Harnack}
Agnid Banerjee, Nicola Garofalo, Isidro~H. Munive, and Duy-Minh Nhieu.
\newblock The {H}arnack inequality for a class of nonlocal parabolic equations.
\newblock {\em Commun. Contemp. Math.}, 23(6):Paper No. 2050050, 23, 2021.

\bibitem[BS24]{BS2024calderon}
Agnid Banerjee and Soumen Senapati.
\newblock The {C}alder\'on problem for space-time fractional parabolic
  operators with variable coefficients.
\newblock {\em SIAM J. Math. Anal.}, 56(4):4759--4810, 2024.

\bibitem[CGRU23]{CGRU2023reduction}
Giovanni Covi, Tuhin Ghosh, Angkana R{\"u}land, and Gunther Uhlmann.
\newblock A reduction of the fractional {C}alder\'on problem to the local
  {C}alder\'on problem by means of the {C}affarelli-{S}ilvestre extension.
\newblock {\em arXiv preprint arXiv:2305.04227}, 2023.

\bibitem[CLL19]{CLL2017simultaneously}
Xinlin Cao, Yi-Hsuan Lin, and Hongyu Liu.
\newblock Simultaneously recovering potentials and embedded obstacles for
  anisotropic fractional {S}chr\"{o}dinger operators.
\newblock {\em Inverse Probl. Imaging}, 13(1):197--210, 2019.

\bibitem[CLR20]{cekic2020calderon}
Mihajlo Cekic, Yi-Hsuan Lin, and Angkana R{\"u}land.
\newblock The {C}alder{\'o}n problem for the fractional {S}chr{\"o}dinger
  equation with drift.
\newblock {\em Cal. Var. Partial Differential Equations}, 59(91), 2020.

\bibitem[CMRU22]{CMRU20}
Giovanni Covi, Keijo M\"{o}nkk\"{o}nen, Jesse Railo, and Gunther Uhlmann.
\newblock The higher order fractional {C}alder\'{o}n problem for linear local
  operators: {U}niqueness.
\newblock {\em Adv. Math.}, 399:Paper No. 108246, 2022.

\bibitem[Dav90]{HeatKernelsSpectralTheory}
Edward~Brian Davies.
\newblock {\em Heat kernels and spectral theory}, volume~92 of {\em Cambridge
  Tracts in Mathematics}.
\newblock Cambridge University Press, Cambridge, 1990.

\bibitem[DNPV12]{DNPV12}
Eleonora Di~Nezza, Giampiero Palatucci, and Enrico Valdinoci.
\newblock Hitchhiker's guide to the fractional {S}obolev spaces.
\newblock {\em Bull. Sci. Math.}, 136(5):521--573, 2012.

\bibitem[Fei24]{Fei24_TAMS}
Ali Feizmohammadi.
\newblock Fractional {C}alder\'on problem on a closed {R}iemannian manifold.
\newblock {\em Trans. Amer. Math. Soc.}, 377(4):2991--3013, 2024.

\bibitem[FGKU25]{feizmohammadi2021fractional}
Ali Feizmohammadi, Tuhin Ghosh, Katya Krupchyk, and Gunther Uhlmann.
\newblock Fractional anisotropic {C}alder\'on problem on closed {R}iemannian
  manifolds.
\newblock {\em J. Differential Geom.}, 131(2):--, 2025.

\bibitem[FKU24]{FKU24}
Ali Feizmohammadi, Katya Krupchyk, and Gunther Uhlmann.
\newblock Calder\'{o}n problem for fractional {S}chr\"{o}dinger operators on
  closed {R}iemannian manifolds.
\newblock {\em arXiv preprint arXiv:2407.16866}, 2024.

\bibitem[FL24]{FL24}
Ali Feizmohammadi and Yi-Hsuan Lin.
\newblock Entanglement principle for the fractional {L}aplacian with
  applications to inverse problems.
\newblock {\em arXiv preprint arXiv:2412.13118}, 2024.

\bibitem[GLX17]{GLX}
Tuhin Ghosh, Yi-Hsuan Lin, and Jingni Xiao.
\newblock The {C}alder\'{o}n problem for variable coefficients nonlocal
  elliptic operators.
\newblock {\em Comm. Partial Differential Equations}, 42(12):1923--1961, 2017.

\bibitem[Gri95]{Grigoryan_upper_heat_kernel}
Alexander Grigoryan.
\newblock Upper bounds of derivatives of the heat kernel on an arbitrary
  complete manifold.
\newblock {\em J. Funct. Anal.}, 127(2):363--389, 1995.

\bibitem[GRSU20]{GRSU20}
Tuhin Ghosh, Angkana R\"{u}land, Mikko Salo, and Gunther Uhlmann.
\newblock Uniqueness and reconstruction for the fractional {C}alder\'{o}n
  problem with a single measurement.
\newblock {\em J. Funct. Anal.}, 279(1):108505, 42, 2020.

\bibitem[GSU20]{GSU20}
Tuhin Ghosh, Mikko Salo, and Gunther Uhlmann.
\newblock The {C}alder\'{o}n problem for the fractional {S}chr\"{o}dinger
  equation.
\newblock {\em Anal. PDE}, 13(2):455--475, 2020.

\bibitem[GU21]{GU2021calder}
Tuhin Ghosh and Gunther Uhlmann.
\newblock The {C}alder\'{o}n problem for nonlocal operators.
\newblock {\em arXiv:2110.09265}, 2021.

\bibitem[HL19]{harrach2017nonlocal-monotonicity}
Bastian Harrach and Yi-Hsuan Lin.
\newblock Monotonicity-based inversion of the fractional {S}chr\"{o}dinger
  equation {I}. {P}ositive potentials.
\newblock {\em SIAM J. Math. Anal.}, 51(4):3092--3111, 2019.

\bibitem[HL20]{harrach2020monotonicity}
Bastian Harrach and Yi-Hsuan Lin.
\newblock Monotonicity-based inversion of the fractional {S}ch\"{o}dinger
  equation {II}. {G}eneral potentials and stability.
\newblock {\em SIAM J. Math. Anal.}, 52(1):402--436, 2020.

\bibitem[Lin24]{lin2024fractional}
Yi-Hsuan Lin.
\newblock The fractional anisotropic {C}alder\'on problem for a nonlocal
  parabolic equation on closed {R}iemannian manifolds.
\newblock {\em arXiv preprint arXiv:2410.17750}, 2024.

\bibitem[LL22]{LL2020inverse}
Ru-Yu Lai and Yi-Hsuan Lin.
\newblock Inverse problems for fractional semilinear elliptic equations.
\newblock {\em Nonlinear Anal.}, 216:Paper No. 112699, 21, 2022.

\bibitem[LL25]{LL25_Integro}
Yi-Hsuan Lin and Hongyu Liu.
\newblock {\em Inverse {P}roblems for {I}ntegro-differential {O}perators},
  volume 222 of {\em Applied Mathematical Sciences}.
\newblock Springer, Cham, 2025.

\bibitem[LLR20]{LLR2019calder}
Ru-Yu Lai, Yi-Hsuan Lin, and Angkana R\"{u}land.
\newblock The {C}alder\'{o}n problem for a space-time fractional parabolic
  equation.
\newblock {\em SIAM J. Math. Anal.}, 52(3):2655--2688, 2020.

\bibitem[LLU22]{LLU2022para}
Ching-Lung Lin, Yi-Hsuan Lin, and Gunther Uhlmann.
\newblock The {C}alder\'{o}n problem for nonlocal parabolic operators.
\newblock {\em arXiv preprint arXiv:2209.11157}, 2022.

\bibitem[LLU23]{LLU2023calder}
Ching-Lung Lin, Yi-Hsuan Lin, and Gunther Uhlmann.
\newblock The {C}alder\'{o}n problem for nonlocal parabolic operators: {A} new
  reduction from the nonlocal to the local.
\newblock {\em arXiv preprint arXiv:2308.09654}, 2023.

\bibitem[LM72]{LM12}
Jacques-Louis Lions and Enrico Magenes.
\newblock {\em Non-homogeneous boundary value problems and applications. {V}ol.
  {II}}, volume Band 182 of {\em Die Grundlehren der mathematischen
  Wissenschaften}.
\newblock Springer-Verlag, New York-Heidelberg, 1972.
\newblock Translated from the French by P. Kenneth.

\bibitem[LNZ24]{LNZ24}
Yi-Hsuan Lin, Gen Nakamura, and Philipp Zimmermann.
\newblock The {C}alder\'on problem for the {S}chr\"odinger equation in
  transversally anisotropic geometries with partial data.
\newblock {\em arXiv preprint arXiv:2408.08298}, 2024.

\bibitem[LZ23]{LZ2023unique}
Yi-Hsuan Lin and Philipp Zimmermann.
\newblock Unique determination of coefficients and kernel in nonlocal porous
  medium equations with absorption term.
\newblock {\em arXiv preprint arXiv:2305.16282}, 2023.

\bibitem[LZ24]{LZ2024approximation}
Yi-Hsuan Lin and Philipp Zimmermann.
\newblock Approximation and uniqueness results for the nonlocal diffuse optical
  tomography problem.
\newblock {\em arXiv preprint arXiv:2406.06226}, 2024.

\bibitem[McL00]{ML-strongly-elliptic-systems}
William McLean.
\newblock {\em Strongly elliptic systems and boundary integral equations}.
\newblock Cambridge University Press, Cambridge, 2000.

\bibitem[MCSA01]{CarracedoSanz2001}
Celso Mart\'inez~Carracedo and Miguel Sanz~Alix.
\newblock {\em The theory of fractional powers of operators}, volume 187 of
  {\em North-Holland Mathematics Studies}.
\newblock North-Holland Publishing Co., Amsterdam, 2001.

\bibitem[RS18]{ruland2018exponential}
Angkana R\"{u}land and Mikko Salo.
\newblock Exponential instability in the fractional {C}alder\'{o}n problem.
\newblock {\em Inverse Problems}, 34(4):045003, 21, 2018.

\bibitem[RS20]{RS20}
Angkana R\"{u}land and Mikko Salo.
\newblock The fractional {C}alder\'{o}n problem: low regularity and stability.
\newblock {\em Nonlinear Anal.}, 193:111529, 56, 2020.

\bibitem[R{\"u}l25]{ruland2023revisiting}
Angkana R{\"u}land.
\newblock Revisiting the anisotropic fractional {C}alder\'on problem.
\newblock {\em Int. Math. Res. Not. IMRN}, (5):Paper No. rnaf036, 28, 2025.

\bibitem[ST17]{ST17_para}
Pablo~Ra\'ul Stinga and Jos\'e{}~L. Torrea.
\newblock Regularity theory and extension problem for fractional nonlocal
  parabolic equations and the master equation.
\newblock {\em SIAM J. Math. Anal.}, 49(5):3893--3924, 2017.

\end{thebibliography}
	
	\bibliographystyle{alpha}
	
\end{document}